\documentclass[english,reqno]{amsart}

\usepackage[margin=2cm]{geometry}
\usepackage{amsmath,amssymb,amsfonts,mathtools}
\usepackage{xcolor}
\usepackage{tikz}
\usepackage{hyperref}
\usepackage{cleveref}
\hypersetup{hidelinks}

\title[Stable solutions of the Allen--Cahn equation in dimension three]{Stable solutions of the Allen--Cahn equation\\ in dimension three are one-dimensional}
\author[H. Chan]{Hardy Chan}
\address{Departement Mathematik und Informatik, Universit\"at Basel,
Spiegelgasse 1, 4051 Basel, Switzerland}
\email{hardy.chan@unibas.ch}

\author[X. Fern\'andez-Real]{Xavier Fern\'andez-Real}
\address{EPFL SB, Station 8, 1015 Lausanne, Switzerland}
\email{xavier.fernandez-real@epfl.ch}

\author[A. Figalli]{Alessio Figalli}
\address{Department of Mathematics, ETH Z\"urich,
R\"amistrasse 101, 8092 Z\"urich, Switzerland}
\email{alessio.figalli@math.ethz.ch}

\author[E. Florit-Simon]{Enric Florit-Simon}
\address{Department of Mathematics, ETH Z\"urich,
R\"amistrasse 101, 8092 Z\"urich, Switzerland}
\email{enric.florit@math.ethz.ch}

\author[J. Serra]{Joaquim Serra}
\address{Department of Mathematics, ETH Z\"urich,
R\"amistrasse 101, 8092 Z\"urich, Switzerland}
\email{joaquim.serra@math.ethz.ch}
\date{}



\theoremstyle{plain}
\newtheorem{thm}{Theorem}[section]
\newtheorem{lem}[thm]{Lemma}
\newtheorem{cor}[thm]{Corollary}
\newtheorem{prop}[thm]{Proposition}

\newtheorem*{conjecture}{Conjecture}

\theoremstyle{definition}
\newtheorem{defn}[thm]{Definition}

\theoremstyle{remark}
\newtheorem{remark}[thm]{Remark}
\newtheorem*{remark*}{Remark}

\newcommand{\cD}{\mathcal{D}}

\crefname{thm}{Theorem}{Theorems}
\crefname{lem}{Lemma}{Lemmas}
\crefname{cor}{Corollary}{Corollaries}
\crefname{prop}{Proposition}{Propositions}
\crefname{claim}{Claim}{Claims}
\crefname{conj}{Conjecture}{Conjectures}
\crefname{prob}{Problem}{Problems}
\crefname{defn}{Definition}{Definitions}
\crefname{notation}{Notation}{Notations}
\crefname{remark}{Remark}{Remarks}

\numberwithin{equation}{section}

\crefname{assumption}{Assumption}{Assumptions}
\begin{document}
\begin{abstract}
We prove that every bounded stable solution of the Allen--Cahn equation in $\mathbb R^3$ is one-dimensional. As a consequence, this shows the validity of De Giorgi's conjecture for monotone solutions in dimension four.
\end{abstract}
\maketitle
\section{Introduction}
The Allen--Cahn equation is a fundamental model for phase transitions, whose stationary form reads as
\begin{equation}\label{eq:allen-cahn}
 -\Delta u+W'(u)=0,\qquad W(t)=\tfrac14(1-t^2)^2.
\end{equation}
Originating in the study of interfaces in binary alloys~\cite{AllenCahn1979}, it has become a central link between nonlinear elliptic equations and the geometry of minimal surfaces~\cite{MM1977,Modica1987,Sternberg1988,CC1995,HT2000,PR2003,Ton2005,TW2012}.

Its variational structure is described by the energy
\[
 E(u;\Omega):=\int_\Omega
       \left(\frac12|\nabla u|^2+W(u)\right)\,dx,
\]
where $\Omega\subset\mathbb R^n$ is a bounded domain. The two minima of $W$ correspond to the phases $u=\pm1$, while the level sets of $u$ describe the interfaces between them.

A fundamental question, deeply related to De Giorgi's original conjecture, is how much the \textit{stability} of such a configuration tells us about the geometry of its interfaces. In this paper we classify bounded stable solutions in dimension three, which implies the validity of De Giorgi's conjecture in dimension four as a well-known consequence.

\subsection{De Giorgi's conjecture and its stable version}

De Giorgi's conjecture, proposed in 1978~\cite{DeG1979}, is a phase transition analogue of the Bernstein problem for minimal graphs.

\begin{conjecture}[De Giorgi]
Let $u:\mathbb R^n\to(-1,1)$ be a classical solution of \eqref{eq:allen-cahn} satisfying $\partial_{x_n}u>0$. If $n\leq8$, then $u$ is one-dimensional: there exist $e\in\mathbb S^{n-1}$ and $\phi:\mathbb R\to(-1,1)$ such that
\[
 u(x)=\phi(e\cdot x).
\]
\end{conjecture}
In simple terms, monotonicity in a single direction is expected to force all the level sets to be parallel hyperplanes. The conjecture was proved in dimension two by Ghoussoub and Gui~\cite{GG1998}, and in dimension three by Ambrosio and Cabr\'e~\cite{AC2000}. A major advance in higher dimensions came with the work of Savin~\cite{Savin2009}, who proved the conjecture up to dimension eight under the additional assumption that the solution tends to $\pm1$ at the two ends of the monotone direction. His work also gives the classification of \textit{energy-minimizers} up to dimension seven. The dimension threshold in De Giorgi's conjecture is sharp: del Pino, Kowalczyk and Wei~\cite{DPKW2011} constructed counterexamples in dimensions nine and higher.

To place the conjecture in a broader variational setting, recall that a solution is \emph{stable} if the second variation of the energy is nonnegative, namely,
\begin{equation}\label{eq:stability}
 Q_u(\xi):=\int_{\mathbb R^n}
       \bigl(|\nabla\xi|^2+W''(u)\xi^2\bigr)\,dx\geq0
 \qquad\text{for every }\xi\in C_c^1(\mathbb R^n).
\end{equation}
Since strictly monotone solutions are in particular stable~\cite{AAC2001}, one is naturally led to the \emph{stable De Giorgi conjecture}: every bounded stable solution of \eqref{eq:allen-cahn} should also be one-dimensional in ``low'' dimensions. This conjecture is, in fact, even stronger than suggested here: its validity in dimension $n$ implies De Giorgi’s original conjecture in dimension $n+1$, through the variational arguments of Jerison–Monneau \cite{JM2004} and Savin \cite{Savin2009}. This motivates the following:
\begin{conjecture}[Stable De Giorgi]
Let $u:\mathbb R^n\to(-1,1)$ be a classical solution of \eqref{eq:allen-cahn} satisfying \eqref{eq:stability}. If $n\leq7$, then $u$ is one-dimensional.
\end{conjecture}
As discussed above, a complete resolution of the stable De Giorgi conjecture (that is, up to dimension seven) would prove De Giorgi's original conjecture up to dimension eight. Previously, the stable conjecture was known only in dimension two by the work of Alberti, Ambrosio, and Cabr\'e~\cite{AC2000, AAC2001}. In dimension eight, Pacard and Wei~\cite{PW2013} constructed nontrivial stable solutions, and there are even nontrivial minimizers constructed by Liu, Wang and Wei~\cite{LWW2017}.

Let us note that all these questions belong to a broader program asking when stability forces interfaces to be flat. Recent advances include the classification of complete two-sided stable minimal hypersurfaces up to $\mathbb R^7$~\cite{CL2024, CLMS2026, Mazet2024, HLW2026}, and of global classical stable solutions to the Bernoulli problem in dimensions three~\cite{CFFS2026} and four~\cite{FRS2026}. Related rigidity results for phase transition models were obtained for the half-Laplacian in~\cite{FigalliSerra2020} and for the free boundary Allen--Cahn problem in~\cite[Section 10]{CFFS2026}. Finally, for the classical Allen--Cahn equation itself,~\cite{FS2025} and~\cite{FS2026} established rigidity results in dimension four under additional energy-growth assumptions.

\subsection{The connection to minimal surfaces}

To understand why such a rigidity statement is natural, it is useful to return to the connection with minimal surfaces, i.e. critical points of the area. Modica and Mortola~\cite{MM1977} made the first link by proving that the appropriately rescaled Allen--Cahn energies $\Gamma$-converge to a multiple of the perimeter. This gives a variational basis for the analogy with the Bernstein problem, and suggests that stable solutions should have a close counterpart in stable minimal hypersurfaces. In dimension three, the geometric picture is then especially compelling: complete two-sided stable minimal surfaces are known to be planes, with no area-growth requirement, by the classical work of do Carmo and Peng~\cite{DP1979}, Fischer-Colbrie and Schoen~\cite{FS1980}, and Pogorelov~\cite{Pog1981}.

This analogy already guided many important developments for the Allen--Cahn equation: the regularity theory for minimizing interfaces plays a central role in Savin's result, while Wang and Wei~\cite{WWFiniteEnds2019,WW2019} and Chodosh and Mantoulidis~\cite{CM2020} developed second-order estimates for stable transition layers.

Nevertheless, passing from this geometric picture to a classification of stable solutions presents substantial difficulties. First of all, without an energy bound, one cannot begin by taking a minimal-surface (blowdown) limit of bounded multiplicity. More deeply, the level sets need not be minimal surfaces, and the stability inequality \eqref{eq:stability} acts in the ambient space. Turning it into a useful inequality on an individual level set is therefore a central step if one wants to exploit the analogy with minimal surfaces. A significant difficulty is that neighboring layers interact even when they are disjoint and nearly flat, and controlling these interactions is essential to the geometric form of the stability inequality which will be described below.

\subsection{Main results}

Our first main result establishes the validity of the stable De Giorgi conjecture in dimension three. Set
\begin{equation*}
 g(t):=\tanh\!\left(\frac{t}{\sqrt2}\right)
\end{equation*}
for the canonical monotone one-dimensional solution to Allen--Cahn, or heteroclinic.
\begin{thm}\label{thm:conditional-classification}
Let $u:\mathbb R^3\to(-1,1)$ be a classical solution of the Allen--Cahn equation \eqref{eq:allen-cahn} satisfying the stability inequality \eqref{eq:stability}. Then, there exist $e\in\mathbb S^2$ and $t_0\in\mathbb R$ such that\footnote{Every bounded solution takes values in $[-1,1]$. Moreover, if either endpoint is attained, the solution is actually constant. Thus, the only other bounded stable solutions in $\mathbb R^3$ are $u\equiv\pm1$.}
\[
 u(x)= g(e\cdot x-t_0).
\]
\end{thm}
The proof builds on two main previous ingredients:
\begin{itemize}
\item the regularity theory of Wang and Wei \cite{WW2019} for stable solutions of the Allen--Cahn equation with sufficiently flat transition layers;
\item the proof in \cite{CFFS2026} of the analogue of Theorem~\ref{thm:conditional-classification} for the free-boundary Allen--Cahn problem.
\end{itemize}

Compared to the work in \cite{CFFS2026}, the main missing ingredient for proving Theorem~\ref{thm:conditional-classification} was a sufficiently precise reduction of the semilinear stability inequality to a geometric stability inequality on the zero level set, which is now provided by Proposition~\ref{prop:bad-ball-stability}. A closely related reduced stability inequality appeared in \cite{CFFS2026}: there, distinct local components of the transition region do not interact through the PDE, while the classification results for the Bernoulli problem provide the required local curvature estimates. Instead, in the current semilinear Allen--Cahn equation, disjoint transition layers interact through their exponential tails, and at the sharp logarithmic separation these interactions can be of the same order as the curvature terms in the reduced stability inequality. For this reason, controlling the errors produced when separating the layers requires the full strength of the Wang--Wei theory, including optimal separation estimates, curvature estimates, and precise approximation by multi-layer Ansätze.

As discussed before, Theorem \ref{thm:conditional-classification} implies De Giorgi's original conjecture in one dimension higher, giving the following:

\begin{cor}[De Giorgi's conjecture in dimension four]
\label{cor:de-giorgi-four}
Let $u:\mathbb R^4\to(-1,1)$ be a classical solution of \eqref{eq:allen-cahn} satisfying $\partial_{x_4}u>0$. Then, there exist $e\in\mathbb S^3$, with $e_4>0$, and $t_0\in\mathbb R$ such that
\[
 u(x)=g(e\cdot x-t_0).
\]
\end{cor}

\begin{proof}
Since the limiting profiles $ u^\pm(x'):=\lim_{x_4\to\pm\infty}u(x',x_4) $ ($x' \in \mathbb R^3$) are stable solutions of \eqref{eq:allen-cahn} in $\mathbb R^3$, Theorem~\ref{thm:conditional-classification} implies that each function is either a constant $\pm1$ or a one-dimensional profile, and therefore minimizes the energy against compactly supported perturbations. Hence, the local minimality theorem of Alberti, Ambrosio, and Cabr\'e~\cite[Theorem~4.4]{AAC2001} implies minimality of $u$ among competitors lying between $u^-$ and $u^+$, and since the limiting profiles $u^\pm$ are themselves minimizers, replacing an arbitrary competitor $v$ with $\max\{u^-,\min\{v,u^+\}\}$ cannot increase its energy. Thus $u$ is a local minimizer, and Savin's classification~\cite[Theorem~2.3]{Savin2009} implies that it is one-dimensional.
\end{proof}

\subsection{Related works}
\label{subsec:related-works-proposal}

Geometric forms of the stability inequality go back to Sternberg and Zumbrun~\cite{SZ98} and underlie the symmetry results of Farina, Sciunzi and Valdinoci~\cite{FSV2008}, providing direct control of level-set curvature through stability. The connection between stability and geometry of interfaces also appears in the minimal surface limit. For sequences with bounded energy, Hutchinson and Tonegawa~\cite{HT2000} established convergence to stationary integral varifolds (i.e. generalized minimal surfaces). Tonegawa~\cite{Ton2005} derived a geometric stability inequality for such limits, and Tonegawa and Wickramasekera~\cite{TW2012} developed their regularity theory using the stronger understanding of the area problem (in particular, a regularity theory for stable varifolds by Wickramasekera~\cite{wic}).

A complementary problem is to obtain quantitative estimates for the transition layers themselves. Early results include the uniform first-order regularity of minimizing layers under a Lipschitz hypothesis, due to Caffarelli and C\'ordoba~\cite{CC2006}. For configurations with several layers, their mutual interaction creates an additional difficulty; its identification and description through a Toda system already appeared in the constructions of del Pino, Kowalczyk and Wei~\cite{DPKW2008}. A major advance in this direction is the work of Wang and Wei~\cite{WWFiniteEnds2019,WW2019}, who obtained second-order estimates for stable transition layers while accounting for these delicate interactions.

The significance of Wang--Wei's analysis extends well beyond the original questions addressed in those works: building on Wang and Wei’s work~\cite{WWFiniteEnds2019}, Chodosh and Mantoulidis~\cite{CM2020} established curvature and sheet-separation estimates on three-manifolds, with applications to multiplicity and Morse index in the minimal-surface limit and a well-known conjecture of Yau. Also, Wang--Wei's estimates play a central role in global classification results under energy-growth assumptions: under cubic energy growth in $\mathbb R^4$, Florit-Simon and Serra~\cite{FS2025} proved one-dimensionality of stable solutions, and Florit-Simon~\cite{FS2026} extended this classification to finite-index solutions.

Alongside these rigidity results, another line of work concerns the geometry of finite-index solutions and their ends. On the constructive side, del Pino, Kowalczyk and Wei~\cite{DPKW2013} obtained solutions near dilations of nondegenerate embedded minimal surfaces of finite total curvature in $\mathbb R^3$, with the same Morse index. On the regularity side, Wang and Wei~\cite{WWFiniteEnds2019} proved that finite Morse index in $\mathbb R^2$ implies finitely many ends and linear energy growth. In $\mathbb R^3$, Florit-Simon~\cite{FS2026} described the ends of finite-index solutions with quadratic energy growth. Under the same assumption, Liu, Wang, Wei and Wu~\cite{LWWW2026} bounded the number of ends in terms of the index and proved axial symmetry for index-one solutions.

\subsection{History of this work and use of artificial intelligence}
\label{sec:ai-declaration}

This work builds on several years of research on the stable Allen--Cahn conjecture. The geometric strategy was developed in the free boundary setting in~\cite{CFFS2026}, where distinct transition layers do not interact. Using these ideas and results, we subsequently worked on the classical Allen--Cahn equation, for which controlling the interactions between neighboring layers was the principal remaining difficulty.

While studying the model case of the classical Allen--Cahn equation in which the bad set is contained in a single ball (see Definition~\ref{defn:bad-set}), we explained our approach to ChatGPT~6 Pro, including our expectation that geometric errors should improve as curvature decreases, as in the free boundary problem. Following these discussions, on 5~September 2026 ChatGPT proposed a preliminary version of Proposition~\ref{prop:bad-ball-stability} for the general case, with interaction errors summable using the crude cubic area bound. This estimate builds on the quantitative analysis of interacting layers in Wang--Wei~\cite{WW2019} and was the only mathematical input from AI that we later used.

At that point, it was already clear to us how to conclude the proof. We then developed a more efficient proof of the estimate and completed the classification using the geometric strategy of~\cite{CFFS2026}. In preparing the manuscript, we also used AI to help write and revise the text.

\subsection{An independent work}
While finishing writing this paper, we learned of an independent proof of the stable De Giorgi conjecture in $\mathbb R^3$ by Liu, Luo, Wang, J.~Wei, Y.~Wei, and Wu~\cite{LiuEtAl2026}.

\subsection{Organization of the paper}
The paper is organized as follows. In Section~\ref{sec:preliminaries} we recall geometric stability and the Wang--Wei estimates. Section~\ref{sec:bounded-cluster} constructs the bad set and proves the presence of arbitrarily wide clean annuli. Section~\ref{sec:bad-ball-stability} derives the geometric stability inequality, which yields the classification of stable solutions in Section~\ref{sec:classification}. Finally, Appendix~\ref{sec:intrinsic-area-proof} proves the auxiliary area--curvature estimate used in the classification argument.

\subsection{Acknowledgements}

H.C. was supported by the Swiss National Science Foundation (SNSF) under grant PZ00P2\_202012. X.F. was supported by the SNSF under grant PZ00P2\_208930, by the Swiss State Secretariat for Education, Research and Innovation (SERI) under contract MB22.00034, and by the Spanish AEI project PID2024-156429NB-I0. E.F. and J.S. were supported by the European Research Council (ERC) under Grant Agreement No.~948029.

\section{Preliminaries}
\label{sec:preliminaries}

We begin by presenting some preliminary results and consequences that will be useful throughout the proof.

\subsection{Notation} 
In this manuscript, unless stated otherwise, $u$ always  denotes a bounded entire stable solution of \eqref{eq:allen-cahn} in $\mathbb R^3$. We denote
\[
\sigma_0:=\int_{\mathbb R}g'(t)^2\,dt=\frac{2\sqrt2}3,\qquad \Sigma_\lambda=\{u=\lambda\}.
\]
We also define
\begin{equation*}
 \mathcal A(u)^2=
 \frac{|D^2u|^2-\bigl|\nabla|\nabla u|\bigr|^2}{|\nabla u|^2}
 \quad\text{on }\{\nabla u\ne0\},\qquad
 \mathcal A=0\quad\text{on }\{\nabla u=0\}.
\end{equation*}
On a regular level set, with normal $\nu=\nabla u/|\nabla u|$ and $\mathrm{II}(X,Y)=-\langle D_X\nu,Y\rangle$, one has
\begin{equation*}
 \mathcal A^2=|\mathrm{II}_{\Sigma_\lambda}|^2
       +|\nabla_{\Sigma_\lambda}\log|\nabla u||^2.
\end{equation*}
Moreover,
\begin{equation*}
 \mathcal A^2|\nabla u|^2
       =|D^2u|^2-\bigl|\nabla|\nabla u|\bigr|^2
       \qquad\text{a.e.}
\end{equation*}
We denote by $H=\operatorname{tr}\mathrm{II}$ the mean curvature. For real matrices $A$ and $B$ of the same size, $A:B:=\sum_{i,j}A_{ij}B_{ij}=\operatorname{tr}(A^{\mathsf T}B)$ denotes their Frobenius inner product.

\subsection{Some known tools}
There is a standard form of the stability inequality in terms of $\mathcal A$:
\begin{lem}[Sternberg--Zumbrun inequality \cite{SZ98}]
\label{lem:sternberg-zumbrun}
Let $U\subset\mathbb R^3$ be open and let $u$ be a classical stable solution of \eqref{eq:allen-cahn} in $U$. For every $\varphi\in C_c^{0,1}(U)$,
\begin{equation*}
 \int_U\mathcal A^2|\nabla u|^2\varphi^2\leq\int_U |\nabla u|^2|\nabla\varphi|^2.
\end{equation*}
\end{lem}

We also recall the standard regularity estimates for classical solutions to the Allen--Cahn equation:
\begin{lem}[Uniform regularity and decay]
\label{lem:uniform-regularity-phase-decay}
Every bounded entire classical solution of \eqref{eq:allen-cahn} satisfies
\begin{equation}\label{eq:uniform-regularity-bounds}
 |u|\leq1,\quad |\nabla u|\leq1/\sqrt{2},\quad
 |D^m u|\leq C_m\quad\text{for}\quad m\geq2,
\end{equation}
where $C_m$ depends only on $m$. There are universal constants $c,C>0$ such that
\begin{equation*}
 1-|u(x)|+|\nabla u(x)|+|D^2u(x)|\leq Ce^{-c\operatorname{dist}(x,\Sigma_0)}
 \quad\text{for all}\quad x\in\mathbb R^3.
\end{equation*}
In particular, a bounded entire solution without zeros is identically $1$ or $-1$.
\end{lem}
\begin{proof}
The uniform bounds follow from \cite[Lemma~2.1]{FS2025} and interior elliptic estimates. The exponential decay follows from the proof of \cite[Lemma~4.2]{KLP2012}, which applies equally in $\mathbb R^3$.
\end{proof}

\subsection{Estimates from Wang--Wei}
\label{subsec:wang-wei-estimates}
Let us present some consequences of the results by Wang--Wei \cite{WW2019} (with the appropriate normalization due to our choice of potential in \eqref{eq:allen-cahn}, which differs by a constant factor from \cite{WW2019}). In the following, $B_r'$ denotes the ball of radius $r> 0$ in ${\mathbb R}^2$.

\begin{thm}[Wang--Wei sheeting and profile estimates]
\label{lem:wang-wei-profiles}
There are universal constants $c,c_\circ,C,\bar\delta>0$, and $R_*\geq4$, such that the following holds.

Let $R\geq R_*$ and let $u:B_{2R}\to(-1,1)$ be a stable solution of \eqref{eq:allen-cahn}. Suppose that for every $x\in B_{3R/2}$ there are $e_x\in\mathbb S^2$ and $a_x\in\mathbb R$ such that
\begin{equation}\label{eq:ww-approximation-hypothesis}
 \sup_{y\in B_2(x)}
 |u(y)-g(e_x\cdot y+a_x)|\leq\bar\delta.
\end{equation}
Then, $|\nabla u|\geq c_\circ$ on $\Sigma_0\cap B_R$, so $\Sigma_0$ is smooth there.

For each $x_0\in\Sigma_0\cap B_R$, there are orthonormal coordinates $(x',s)$ centered at $x_0$ and a consecutive family of smooth ordered zero graphs $f_1<\cdots<f_Q$ on $B'_{8cR}$, contained in $B_{R/2}(x_0)$, with the following properties: let
\begin{equation*}
 \Gamma_j=\{(x',f_j(x')):x'\in B'_{4cR}\}.
\end{equation*}
Then:
\begin{enumerate}
\renewcommand{\labelenumi}{(\roman{enumi})}
\renewcommand{\theenumi}{\roman{enumi}}

\item\label{item:ww-sheet-geometry} \emph{Graph representation and geometric estimates.}
\begin{equation}\label{eq:ww-graph-bounds}
 \begin{split}
 &\Sigma_0\cap B_{2cR}(x_0)
 =\bigcup_{j=1}^Q\bigl(\Gamma_j\cap B_{2cR}(x_0)\bigr),\qquad  \|Df_j\|_\infty
  +R\|\mathrm{II}_{\Gamma_j}\|_{C^{0,1}}
  +R^2\|H_{\Gamma_j}\|_{C^{0,1/2}}\leq C.
 \end{split}
\end{equation}

\item\label{item:ww-sheet-separation} \emph{Sheet separation.}
\begin{equation}\label{eq:ww-sheet-separation}
 \operatorname{dist}(y,\Gamma_k)\geq\sqrt2\log R-C
 \quad\text{for }y\in\Gamma_j\cap B_{2cR}(x_0),\quad k\ne j.
\end{equation}

\item\label{item:ww-profile-approximation} \emph{Profile approximation.} Let $d_j$ be the signed distance to the extended $j$-th graph, positive above it. Each $d_j$ is smooth in $B_{cR}(x_0)$, and there is $\sigma\in\{-1,1\}$ such that the unshifted profile sum satisfies
\begin{equation}\label{eq:ww-profile-inputs-proof}
 \begin{split}
 &\|u-G_0\|_{C^2(B_{cR}(x_0))}\leq CR^{-2},\qquad \sigma G_0(x)=
  \sum_{j=1}^Q(-1)^{j-1}g(d_j(x))
  -\frac{1+(-1)^Q}{2}.
 \end{split}
\end{equation}
\end{enumerate}
\end{thm}
\begin{proof}
By \cite[Corollary~1.3, p.~4, and its proof, p.~71]{WW2019}, using the stable Bernstein theorem for minimal surfaces in dimension three,
\begin{equation}\label{eq:ww-enhanced-curvature}
 |\nabla u|\geq c_\circ,\qquad \mathcal A(u)\leq CR^{-1}
 \quad\text{in }B_{R/4}(x_0)\cap\{|u|\leq9/10\}.
\end{equation}
This means that $u$ satisfies the ``sheeting assumptions'' required for the Wang--Wei estimates. Then \cite[Lemma~2.2, (3.3) and Lemma~3.1, (11.1) and the display below, and Proposition~10.1]{WW2019} give \eqref{eq:ww-graph-bounds}, and \cite[Proposition~10.1]{WW2019} corresponds to \eqref{eq:ww-sheet-separation}.

For \eqref{eq:ww-profile-inputs-proof}, let $G_h$ denote the shifted, truncated profile sum constructed in \cite[(4.3) and Proposition~4.1]{WW2019}, expressed in our original coordinates. (Essentially, this is the profile in \eqref{eq:ww-profile-inputs-proof} but with shifts $h_j(x)$ in the normal direction, and truncated far from the layers.) By \cite[(11.1) and the following display, p.~67]{WW2019} and \cite[Lemma~4.6]{WW2019},
\begin{equation}\label{eq:ww-shifted-profile-inputs}
 \|u-G_h\|_{C^2(B_{cR}(x_0))}
 +\sup_j\|h_j\|_{C^2(\Gamma_j)} \leq CR^{-2}.
\end{equation}
Next, we note that the errors from removing the truncations and the shifts are summable, uniformly in the number of sheets. Indeed, \eqref{eq:ww-sheet-separation} bounds uniformly the number of sheets within distance $C\log R$ of any fixed point, while the contributions of the remaining sheets and their first two derivatives are summable by exponential decay. Thus the cutoff construction in \cite[Section~4.1]{WW2019} and \eqref{eq:ww-shifted-profile-inputs} give
\[
 \|G_h-G_0\|_{C^2(B_{cR}(x_0))}\leq CR^{-2},
\]
with $C$ independent of $Q$. Combining these estimates proves \eqref{eq:ww-profile-inputs-proof}.
\end{proof}
A ball satisfying \eqref{eq:ww-approximation-hypothesis} is considered ``good'' (cf. \emph{clean balls} below), since $\{u=0\}$ is smooth with estimates there. The condition \eqref{eq:ww-approximation-hypothesis} is readily connected to the stability inequality; here is the form we will actually use:

\begin{lem}
\label{lem:small-curvature-heteroclinic}
For every $\varepsilon>0$ there is $\delta>0$ such that, if a bounded entire stable solution $u$ of \eqref{eq:allen-cahn} satisfies
\[
 \int_{B_2}\mathcal A^2|\nabla u|^2\leq\delta,
\]
then, for some $e\in\mathbb S^2$ and $t_0\in\mathbb R$,
\[
 \|u(x)-g(e\cdot x-t_0)\|_{C^2(B_2)}<\varepsilon.
\]
\end{lem}

\begin{proof}
Fix $\varepsilon>0$. By compactness (as in \cite[proof of Theorem~3.4]{FS2025}), there exists $\lambda>0$ such that every solution with $|u(0)|>1-\lambda$ is within $\varepsilon/2$ of one of the constants $\pm1$ in $C^2(B_2)$. Since these constants are limits of translated heteroclinics, the conclusion follows in this case.

If the conclusion fails in the remaining case as $\delta\downarrow0$, compactness yields a stable limit $u_\infty$ in $C^2_{\mathrm{loc}}(\mathbb R^3)$ with $|u_\infty(0)|\leq1-\lambda$. Then Fatou's lemma on $B_2\cap\{\nabla u_\infty\ne0\}$ gives vanishing curvature, hence $\mathcal A(u_\infty)\equiv0$ in $B_2$. But then (see \cite[Lemma~2.6 and Proposition~2.8]{FS2025}) $u_\infty$ is a translated heteroclinic, contradicting the failure of approximation in $C^2(B_2)$.
\end{proof}

\section{The bad set and the bounded cluster}
\label{sec:bounded-cluster}
Theorem~\ref{lem:wang-wei-profiles} and Lemma~\ref{lem:small-curvature-heteroclinic} readily motivate the notion of a ``bad ball'', which we now introduce, following \cite{CFFS2026}. The main idea is:
\begin{itemize}
\item $\{u=0\}$ behaves more and more like a minimal surface the further one goes from the bad balls; see \eqref{eq:ww-graph-bounds}.
\item The bad set is controlled via stability, by Lemma~\ref{lem:sternberg-zumbrun}.
\end{itemize}

Fix once and for all the universal $\delta>0$ supplied by Lemma~\ref{lem:small-curvature-heteroclinic} with $\varepsilon=\min\{\bar\delta/2,1/100\}$, where $\bar\delta$ is from Theorem~\ref{lem:wang-wei-profiles}. This choice is independent of $u$.

\begin{defn}[Bad set, clean ball]\label{defn:bad-set}
For a bounded entire stable solution $u$ of \eqref{eq:allen-cahn}, set
\begin{equation}\label{eq:bad-set}
 \mathcal X=\bigg\{x:\int_{B_2(x)}\mathcal A^2|\nabla u|^2>\delta\bigg\}.
\end{equation}
Choose a maximal set $\mathcal Z\subset\mathcal X$ with $|z-z'|\geq1$ for distinct $z,z'\in\mathcal Z$. A ball $B_R(x)$ is \textit{clean} if $B_{2R}(x)\cap\mathcal X=\varnothing$.
\end{defn}

By Lemma~\ref{lem:small-curvature-heteroclinic} and the fixed choice of $\delta$, every clean ball with $R\geq R_*$ satisfies \eqref{eq:ww-approximation-hypothesis}. At a zero outside $\mathcal X$, the same $C^2$ approximation gives a uniformly positive derivative in the profile direction throughout its half-unit ball. Thus, after decreasing $c_\circ$ if necessary,
\begin{equation}\label{eq:clean-gradient-bound}
 |\nabla u|\geq c_\circ\quad\text{on }\Sigma_0\setminus\mathcal X.
\end{equation}
Moreover, $\Sigma_0\cap B_{1/2}(y)$ is a graph with uniformly bounded slope for each $y\in\Sigma_0\setminus\mathcal X$. The set $\mathcal Z$ is locally finite. Its centered unit balls cover $\mathcal X$, and the doubled balls have overlap at most $5^3$.

\begin{remark}\label{rem:bad-set-nonempty}
For this fixed $\delta$, $\mathcal X=\varnothing$ implies that $u$ is one-dimensional. Indeed, since we can apply Lemma~\ref{lem:small-curvature-heteroclinic} around every $x\in\mathbb R^3$, the hypotheses of Theorem~\ref{lem:wang-wei-profiles} hold with $R$ arbitrarily large, and sending $R\to\infty$ shows that $u$ is one-dimensional (see \eqref{eq:ww-enhanced-curvature} for instance).
\end{remark}

Similarly to  \cite[Lemma 10.8]{CFFS2026}, a logarithmic cutoff argument lets us find an arbitrarily wide clean annulus:
\begin{lem}[Existence of a clean annulus]\label{lem:clean-annulus}
Assume $\mathcal X\ne\varnothing$. For every $L\geq1$ and $\Lambda>0$, there exist $z_\Lambda\in\mathcal Z$ and $R_\Lambda>1$ such that
\begin{equation}\label{eq:clean-annulus}
 \operatorname{dist}(x,\mathcal X)\geq L
 \quad\text{for every }x\in B_{R_\Lambda+\Lambda}(z_\Lambda)
                         \setminus B_{R_\Lambda}(z_\Lambda).
\end{equation}
\end{lem}

\begin{proof}
The proof is a variation of  \cite[Proof of Lemma 10.8]{CFFS2026}. First, observe that by Lemmas~\ref{lem:sternberg-zumbrun} and~\ref{lem:uniform-regularity-phase-decay},
\begin{equation}\label{eq:annulus-capacity}
 \int\mathcal A^2|\nabla u|^2\eta^2
 \leq\int|\nabla u|^2|\nabla\eta|^2
 \leq\frac12\int|\nabla\eta|^2
 \quad\text{ for all}\quad\eta\in C_c^{0,1}(\mathbb R^3).
\end{equation}
Taking $\eta=1$ on $B_t(a)$, supported in $B_{2t}(a)$, with $|\nabla\eta|\leq C/t$, and using the overlap bound gives, for $t\geq1$,
\begin{equation}
 \int_{B_t(a)}\mathcal A^2|\nabla u|^2\leq Ct,\qquad
 \delta\,\#(\mathcal Z\cap B_t(a))
 \leq C\int_{B_{t+2}(a)}\mathcal A^2|\nabla u|^2\leq Ct.
 \label{eq:annulus-upper-count}
\end{equation}
Now, if \eqref{eq:clean-annulus} fails, then for every $z\in\mathcal Z$ and $r>1$ there is $q\in\mathcal Z$ with $r-L-1<|q-z|<r+\Lambda+L+1$. Taking disjoint such intervals and using \eqref{eq:annulus-upper-count}, we obtain
\[
 ct\leq\#(\mathcal Z\cap B_t(z))\leq Ct
 \quad\text{ for all}\quad z\in\mathcal Z,\ t\geq t_0,
\]
where $t_0\geq2$ and $c,C>0$ depend only on $L,\Lambda$. These linear bounds allow us to construct a logarithmic cutoff whose Dirichlet energy is smaller than the curvature mass it retains:

Fix $z_0\in\mathcal Z$, take $R$ large, and let
\[
 d_R(x)=\operatorname{dist}\bigg(x,
       \bigcup_{q\in\mathcal Z\cap B_R(z_0)}\overline B_2(q)\bigg).
\]
For $3t_0\leq t\leq R$, a maximal $t$-separated subset of $\mathcal Z\cap B_R(z_0)$ has at most $CR/t$ points, since its disjoint $t/3$-balls each contain at least $ct$ centers, all lying in $B_{2R}(z_0)$. By maximality, the corresponding $3t$-balls cover $\{d_R<t\}$. For $t<3t_0$, we instead cover this set by enlarging each of the original balls $B_2(q)$ to $B_{t+2}(q)$, as there are at most $CR$ such balls. These two coverings give
\[
 |\{d_R<t\}|\leq CR(1+t)^2\quad\text{for}\quad 0<t\leq R.
\]
Since the cutoff $\eta_R=\left(1-\frac{\log(1+d_R)}{\log(1+R)}\right)_+$ equals one on the balls $B_2(q)$ used to define $d_R$, we have
\[
 \int\mathcal A^2|\nabla u|^2\eta_R^2
 \geq5^{-3}\sum_{q\in\mathcal Z\cap B_R(z_0)}
            \int_{B_2(q)}\mathcal A^2|\nabla u|^2
 \geq cR,
\]
whereas the preceding volume bound gives
\[
\begin{aligned}
 \int|\nabla\eta_R|^2
 &\leq\frac1{\log^2(1+R)}
       \int_{\{0<d_R<R\}}\frac{dx}{(1+d_R)^2} \leq\frac1{\log^2(1+R)}
   \left(\frac{|\{d_R<R\}|}{(1+R)^2}
       +2\int_0^R\frac{|\{d_R<t\}|}{(1+t)^3}\,dt\right)
 \leq\frac{CR}{\log(1+R)},
\end{aligned}
\]
contradicting \eqref{eq:annulus-capacity} as $R\to\infty$.
\end{proof}

\begin{defn}[Bounded cluster and its neighborhoods]\label{defn:bounded-cluster}
Assume $\mathcal X\ne\varnothing$, and fix $L_0\geq2$ and $\Lambda\geq32L_0$. Choose $z_\Lambda,R_\Lambda$ by Lemma~\ref{lem:clean-annulus} with $L=8L_0$, and set
\begin{equation*}
 \mathcal X_\Lambda=\mathcal X\cap B_{R_\Lambda}(z_\Lambda),
 \qquad
 \mathcal Z_\Lambda=\mathcal Z\cap B_{R_\Lambda}(z_\Lambda),
 \qquad N=\#\mathcal Z_\Lambda.
\end{equation*}
Define the distance to this bounded cluster and its neighborhoods by
\begin{equation*}
 \cD(x)=\operatorname{dist}(x,\mathcal X_\Lambda),
 \qquad \mathcal S_\alpha=\{x:\cD(x)<\alpha\}\quad(\alpha>0),
\end{equation*}
and let $ \Sigma=\Sigma_0\cap\{L_0<\cD<\Lambda/3\}.$
\end{defn}

Since the annulus in \eqref{eq:clean-annulus} is disjoint from every $B_{8L_0}(z)$ with $z\in\mathcal X$, it follows that
\begin{equation}\label{eq:cluster-radial-bounds}
\begin{aligned}
 |z-z_\Lambda|&\leq R_\Lambda-8L_0
 &&\text{for all}\quad z\in {\mathcal X_\Lambda},\\
 |z-z_\Lambda|&\geq R_\Lambda+\Lambda+8L_0
 &&\text{for all}\quad z\in\mathcal X\setminus\mathcal X_\Lambda.
\end{aligned}
\end{equation}
In particular, $\overline{\mathcal X_\Lambda}$ is nonempty and compact, and $1\leq N<\infty$. Also, every point of $\mathcal X_\Lambda$ lies within distance one of $\mathcal Z_\Lambda$, and by the overlap bound
\begin{equation}\label{eq:cluster-cover-mass}
 \mathcal S_\alpha\subset\bigcup_{z\in\mathcal Z_\Lambda}B_{\alpha+1}(z),\qquad
 \delta N\leq\sum_{z\in\mathcal Z_\Lambda}
     \int_{B_2(z)}\mathcal A^2|\nabla u|^2
 \leq5^3\int_{\mathcal S_2}\mathcal A^2|\nabla u|^2.
\end{equation}
Moreover, every exterior bad point has distance at least $\Lambda+16L_0-\cD(x)>\cD(x)$ from $x$ if $\cD(x)<\Lambda/3$. Consequently,
\begin{equation}\label{eq:cluster-full-distance}
 \operatorname{dist}(x,\mathcal X)=\cD(x)
 \qquad\text{whenever }\cD(x)<\Lambda/3.
\end{equation}
In particular, $B_{\cD(x)/2}(x)\cap\mathcal X=\varnothing$ when $L_0<\cD(x)<\Lambda/3$, so $B_{\cD(x)/4}(x)$ is clean. Thus, $\Sigma$ is a smooth embedded surface by \eqref{eq:clean-gradient-bound}, with
\begin{equation}\label{eq:cluster-curvature-bounds}
 |\mathrm{II}_\Sigma(y)|\leq C\cD(y)^{-1},\qquad
 |H_\Sigma(y)|\leq C\cD(y)^{-2}
 \quad\text{for all}\quad y\in\Sigma.
\end{equation}
These bounds follow from Theorem~\ref{lem:wang-wei-profiles} when $\cD(y)\geq4R_*$, and from \eqref{eq:clean-gradient-bound} and Lemma~\ref{lem:uniform-regularity-phase-decay} at the remaining scales.

\section{A geometric form of stability}
\label{sec:bad-ball-stability}
We work in the setting of Definition~\ref{defn:bounded-cluster}, with $L_0$ sufficiently large. The goal of this section is to derive a geometric stability inequality on the zero set of $u$, away from the bad set, while retaining the positive curvature contribution near the bad set.

\begin{defn}[Normal coordinates and endpoints]
\label{defn:normal-coordinates-endpoints} 
Fix $y\in\Sigma$, and set the truncation distance
\begin{equation*}
 T(y)=3\log \cD(y).
\end{equation*}
Orient the normal vector towards increasing values of $u$, and let $F$ be the associated Fermi coordinate map:
\begin{equation*}
 \nu(y)=\frac{\nabla u(y)}{|\nabla u(y)|},\qquad F(y,t)=y+t\nu(y).
\end{equation*}
Move along each normal ray, stopping when another zero is equally close or at distance $T(y)$, whichever occurs first; the signed endpoints are denoted $b_\pm(y)$. Formally:
\begin{equation}\label{eq:ww-endpoints-definition}
 b_\pm(y)=\pm\sup\left\{s\in[0,T(y)]:
 \begin{array}{l}
 y\text{ is the unique nearest point in }\Sigma_0  
       \text{ to }F(y,\pm t)\\
 \text{for every }0\leq t<s
 \end{array}\right\}.
\end{equation}
Thus $b_\pm(y)=\pm T(y)$ if equidistance is not reached before truncation. The spanned normal interval is
\begin{equation}\label{eq:ww-normal-interval}
 \{F(y,t):b_-(y)<t<b_+(y)\}.
\end{equation}
It is convenient to define its extension by one unit on each end, denoted
\begin{equation}\label{eq:normal-param-manifold}
 \mathcal N=\{(y,t):y\in\Sigma,\ b_-(y)-1<t<b_+(y)+1\}.
\end{equation}
For $L_0$ large, $F$ gives smooth normal coordinates locally on $\mathcal N$ (this is justified below). Note that while the normal intervals from \eqref{eq:ww-normal-interval} have disjoint images, their extensions in $\mathcal N$ overlap, so $F$ need not be globally injective there.

Given $(y,t)\in\mathcal N$, choose an open neighborhood $\Gamma\subset\Sigma$ of $y$ on which these coordinates are injective, with $\Gamma\times\{t\}\subset\mathcal N$, and write
\[
 \Gamma^t=F(\Gamma,t),\qquad \Gamma^0=\Gamma.
\]
All quantities on $\Gamma^t$ below are evaluated at $F(y,t)$, whereas quantities on $\Gamma$ are evaluated at $y$. Ambient functions are composed with $F$; in particular, writing $u(y,t)=u(F(y,t))$, we have $\partial_tu(y,t)=\partial_t(u\circ F)(y, t) = \nabla u(F(y, t))\cdot \nu(y)$.
\end{defn}

We record some auxiliary estimates for future reference.
\begin{lem}[Auxiliary estimates]\label{lem:ww-normal-profiles}
In the setting above, the following hold for $L_0$ sufficiently large and with universal constants.
\begin{enumerate}
\renewcommand{\labelenumi}{(\roman{enumi})}
\renewcommand{\theenumi}{\roman{enumi}}
\item\label{item:ww-normal-intervals} \emph{Fermi coordinates.} For every $y\in\Sigma$, there is an open neighborhood $\Gamma\subset\Sigma$ of $y$ such that $F$ is injective on
\begin{equation}\label{eq:ww-normal-range}
 \Gamma\times[-4\log \cD(y)-1,4\log \cD(y)+1].
\end{equation}
Also, its Jacobian at $(y,t)$ is $1+O(|t|\cD(y)^{-1})$.
\item\label{item:ww-endpoints} \emph{Truncation and endpoint bounds.} The truncation distance satisfies
\begin{equation}\label{eq:ww-truncation-distance}
 |\nabla_\Sigma T(y)|\leq3\cD(y)^{-1}\leq1\quad\text{a.e.}
\end{equation}
The endpoint functions $b_\pm$ in \eqref{eq:ww-endpoints-definition} are locally Lipschitz and satisfy
\begin{equation}\label{eq:ww-endpoint-bounds}
 -T(y)\leq b_-(y)\leq-\frac{\log \cD(y)}2,\qquad
 \frac{\log \cD(y)}2\leq b_+(y)\leq T(y),\qquad
 |\nabla_\Sigma b_-|+|\nabla_\Sigma b_+|\leq C\quad\text{a.e.}
\end{equation} 
\item\label{item:ww-profiles} \emph{Estimates on normal quantities.} For $y\in\Sigma$,
\begin{equation}\label{eq:ww-core-profile}
 |\partial_tu(y,t)-g'(t)|\leq C\cD(y)^{-2}e^{\sqrt2|t|}
 \quad\text{if}\quad |t|\leq\frac{\log \cD(y)}2.
\end{equation}
At an endpoint $F(y,b_\pm(y))$ reached by equidistance, let $\nu_\pm$ be the unit normal of the adjacent neighboring sheet, evaluated at its nearest point. Orient it so that $\nu_\pm\cdot\nu(y)>0$, i.e. opposite to $\nabla u$ on that sheet. Then
\begin{equation}\label{eq:ww-endpoint-angle}
 |\nu(y)-\nu_\pm|\leq C\cD(y)^{-1}
       \bigl(2|b_\pm(y)|-\sqrt2\log \cD(y)+C\bigr).
\end{equation}
\end{enumerate}
\end{lem}

\begin{proof}
Fix $y\in\Sigma$ and set $R=\cD(y)/4$. By \eqref{eq:cluster-full-distance}, $B_R(y)$ is clean. Apply Lemma~\ref{lem:small-curvature-heteroclinic} and Theorem~\ref{lem:wang-wei-profiles} there. Its larger graph charts extend a distance comparable to $\cD(y)$ beyond their interior restrictions. Taking $L_0$ large enough that $4\log \cD(y)+2<cR$ puts the normal segments in \eqref{eq:ww-normal-range} and their unit neighborhoods inside $B_{cR}(y)$, and every competing nearest zero inside $B_{2cR}(y)$. These are charts of $\Sigma_0$ and may extend beyond $\Sigma$.

\noindent\emph{Part (\ref{item:ww-normal-intervals}).}
The injectivity of the normal map $F$ on some neighborhood (in $\Sigma$) of $y$, with normal range \eqref{eq:ww-normal-range}, and the Jacobian bound, are simple consequences of the curvature bounds on clean balls (i.e. \eqref{eq:ww-graph-bounds}). We also have bounded projection derivatives there, see \cite[Lemma~3.1, (3.6), and Lemma~3.5]{WW2019}.

\smallskip
\noindent\emph{Part (\ref{item:ww-endpoints}).}
The bound \eqref{eq:ww-truncation-distance} follows just from the $1$-Lipschitz property of $\cD$. Regarding \eqref{eq:ww-endpoint-bounds}, note that by \eqref{eq:ww-sheet-separation} an equidistant endpoint satisfies $|b_\pm(y)|\geq(\log \cD(y))/\sqrt2-C\geq(\log \cD(y))/2$ for $L_0$ sufficiently large. If $b_\pm(y)$ is a ``truncated'' endpoint instead, i.e. $|b_\pm(y)|=T(y)=3\log \cD(y)$, the bounds on $|b_\pm|$ in \eqref{eq:ww-endpoint-bounds} also hold.

In the normal range \eqref{eq:ww-normal-range}, the starting sheet has unique normal projection. Before the first equidistant endpoint, the upper ray stays between this sheet and its upper neighbor, if present. A segment from this region to any other sheet crosses a closer sheet, so only the adjacent upper sheet can compete. Finally, to prove the derivative bounds in \eqref{eq:ww-endpoint-bounds}, at an upper endpoint let $r$ be the distance to the next sheet, and orient its normal $\nu_+$ in the direction of $\nu(y)$. Near parallelism from \cite[Lemma~3.4]{WW2019} gives
\begin{equation*}
 \partial_t(t-r(F(y,t)))=1+\nu_+\cdot\nu(y)\geq1
 \quad\text{if}\quad t=b_+(y)<T(y).
\end{equation*}
Since $|\nabla r|=1$ and $|DF|\leq C$, the implicit function theorem gives the Lipschitz bound for this endpoint. Taking the minimum with $T$ preserves the bound by \eqref{eq:ww-truncation-distance}. The same argument using the lower neighboring sheet gives the assertion for $b_-$. The nearest-point region along a normal ray is an interval: for each competing zero $q$, the difference $|F(y,t)-q|^2-t^2=|y-q|^2+2t\nu(y)\cdot(y-q)$ is affine in $t$. This proves the asserted description of \eqref{eq:ww-normal-interval}. Their images are disjoint by the uniqueness condition in \eqref{eq:ww-endpoints-definition}.

\smallskip
\noindent\emph{Part (\ref{item:ww-profiles}).}
Let $G_0$ be the approximation of $u$ in \eqref{eq:ww-profile-inputs-proof}, with error $C\cD(y)^{-2}$ in $C^2$. After subtracting their limiting phase values, the profiles from the other sheets, together with their first two derivatives, are bounded in sum by $C\cD(y)^{-2}e^{\sqrt2|t|}$, by \eqref{eq:ww-sheet-separation} and \cite[Lemma~3.6]{WW2019}. This proves \eqref{eq:ww-core-profile}.

Finally, to prove \eqref{eq:ww-endpoint-angle}, consider an upper endpoint reached by equidistance. Set $P=F(y,b_+(y))$, $\ell=b_+(y)\leq3\log \cD(y)$, $Q=y$, and $\nu=\nu(y)$. Let $Q_+$ be the nearest point on the upper neighboring sheet, with normal $\nu_+$ oriented as in the statement. Then
\[
 P=Q+\ell\nu=Q_+-\ell\nu_+.
\]
Set $e=(\nu+\nu_+)/|\nu+\nu_+|$. By near parallelism \cite[Lemma~3.4]{WW2019}, the starting and upper sheets are graphs $f$ and $f_+$ over $e^\perp$ on a disk of radius a fixed small multiple of $\cD(y)$, contained in the larger charts above. Since $Q_+-Q$ is parallel to $e$, their common horizontal coordinate is taken as zero. On this disk,   \eqref{eq:ww-sheet-separation} gives, for a universal $C_0$,
\begin{equation}\label{eq:ww-positive-separation-proof}
 f_+-f-\sqrt2\log \cD(y)+C_0+1\geq1.
\end{equation}
Subtracting the graph mean-curvature equations and using \eqref{eq:ww-graph-bounds} gives
\begin{equation*}
 \operatorname{div}\bigl(a\nabla(f_+-f)\bigr)=H_+-H,\qquad
 |\nabla a|\leq C\cD(y)^{-1},\qquad |H_+-H|\leq C\cD(y)^{-2},
\end{equation*}
where $a$ is uniformly elliptic and $H,H_+$ are the respective mean curvatures with normals having positive $e$-component. Apply the rescaled Harnack and interior gradient estimates as in \cite[proof of Lemma~10.6]{WW2019} to the positive function in \eqref{eq:ww-positive-separation-proof}. Its lower bound by one absorbs the inhomogeneous term and gives
\begin{equation*}
 |\nabla(f_+-f)(0)|\leq C\cD(y)^{-1}
        \bigl(f_+(0)-f(0)-\sqrt2\log \cD(y)+C_0+1\bigr).
\end{equation*}
Since $\nu\cdot e=\nu_+\cdot e$, the bisector coordinates give
\begin{equation*}
 \begin{split}
 f_+(0)-f(0)&=2\ell(\nu\cdot e)\leq2\ell,\\
 |\nu-\nu_+|&=(\nu\cdot e)|\nabla(f_+-f)(0)|
 \leq C\cD(y)^{-1}(2\ell-\sqrt2\log \cD(y)+C_0+1).
 \end{split}
\end{equation*}
This proves \eqref{eq:ww-endpoint-angle} at the upper endpoint. The lower endpoint follows by the same argument using the lower neighboring graph $f_-$, with positive gap $f-f_-$ and mean-curvature difference $H-H_-$.
\end{proof}

\subsection{Tangential and normal estimates}
\label{subsec:layerwise-stability}
We use the normal coordinates and endpoints of Definition~\ref{defn:normal-coordinates-endpoints}.

We first record the equation satisfied by the normal derivative.

\begin{lem}[Differentiated equation in $\partial_t$]
\label{lem:normal-differentiated-equation}
Let $\Gamma$ be a smooth oriented surface with unit normal $\nu$. Wherever $F(y,t)=y+t\nu(y)$ defines smooth normal coordinates, a solution of \eqref{eq:allen-cahn} satisfies
\begin{equation}\label{eq:normal-differentiated-equation}
 \begin{aligned}
 (-\Delta+W''(u))\partial_tu
 &=-|\mathrm{II}_{\Gamma^t}|^2\partial_tu
   +2\mathrm{II}_{\Gamma^t}:\nabla_{\Gamma^t}^2u
   +\nabla_{\Gamma^t}H_{\Gamma^t}\cdot\nabla_{\Gamma^t}u.
 \end{aligned}
\end{equation}
Here $\Gamma^t=F(\Gamma,t)$ is oriented by the normal $\nu(F(y,t))=\nu(y)$, and we use the conventions $\mathrm{II}_{\Gamma^t}=-D\nu|_{T\Gamma^t}$ and $H_{\Gamma^t}=\operatorname{tr}\mathrm{II}_{\Gamma^t}$. The Hessian $\nabla_{\Gamma^t}^2u$ is intrinsic.
\end{lem}

\begin{proof}
The Laplacian in Fermi coordinates is
\[
 \Delta=\partial_{tt}-H_{\Gamma^t}\partial_t+\Delta_{\Gamma^t}.
\]
Moreover, the variation formulas in \cite[Appendix~A, (A.3) and (A.7)]{CM2020}, applied in $\mathbb R^3$ (and with our sign convention) give
\[
 \begin{aligned}
  \partial_tH_{\Gamma^t}&=|\mathrm{II}_{\Gamma^t}|^2,\qquad 
  [\partial_t,\Delta_{\Gamma^t}]u
  =2\mathrm{II}_{\Gamma^t}:\nabla_{\Gamma^t}^2u
    +\nabla_{\Gamma^t}H_{\Gamma^t}\cdot\nabla_{\Gamma^t}u.
 \end{aligned}
\]
Differentiating $\Delta u =W'(u)$ in $t$ then yields  \eqref{eq:normal-differentiated-equation}.
\end{proof}

We next establish tangential derivative and decay estimates throughout the domain $\mathcal N$.

\begin{lem}[Tangential derivative and decay estimates]\label{lem:layerwise-stability}
In the setting fixed at the beginning of this section, after increasing $L_0$ if necessary, the following estimates hold for every $(y,t)\in\mathcal N$ for some universal $C$:
\begin{align}
 |\nabla_{\Gamma^t}u|+|\nabla_{\Gamma^t}^2u|
 &\leq C\cD(y)^{-2},\label{eq:normal-tangential-bounds}\\
 |\nabla_{\Gamma^t}H_{\Gamma^t}|
 &\leq C\cD(y)^{-4/3},\label{eq:normal-mean-curvature-gradient}\\
 |\nabla u|+|D^2u|
 &\leq C\bigl(e^{-\sqrt2|t|}+\cD(y)^{-2}\bigr).
 \label{eq:normal-cell-derivative-tail}
\end{align}
\end{lem}

\begin{proof}
Fix $y\in\Sigma$ and work in a fixed unit neighborhood of the segment $\{F(y,t):0\leq t\leq b_+(y)+1\}$ inside an enlarged chart. Let $r$ be the distance to the upper neighboring sheet, when it is present among the local graphs. Theorem~\ref{lem:wang-wei-profiles} ensures that $r$ is smooth throughout this neighborhood; write $\nu_+=-\nabla r$ for its normal at the nearest point. The profile approximation (which is a consequence of \eqref{eq:ww-profile-inputs-proof}, see below)  we use is
\begin{equation}\label{eq:normal-two-profiles}
 \sum_{k=0}^2\bigl|D^k\bigl(u-g(t)-g(r)+1\bigr)\bigr|\le C\cD(y)^{-2}, \quad
 \nabla r=-\nu_+,
 \quad \partial_t r=-\nu(y)\cdot\nu_+,\quad \mbox{ at } F(y,t).
\end{equation}
Here the derivatives in the sum are ambient Euclidean derivatives at $F(y,t)$, and $t$ denotes the local signed-distance function to the starting sheet. If no upper neighboring sheet is present, the term $g(r)-1$ and its derivatives are omitted. Any sheet entering or leaving the chart is at distance at least $c\cD(y)$ from the point $F(y,t)$ under consideration. Its profile tail and its first two derivatives are therefore exponentially small and can be absorbed into the $C\cD(y)^{-2}$ error.

To obtain \eqref{eq:normal-two-profiles}, apply \eqref{eq:ww-profile-inputs-proof} with $R=\cD(y)/4$. This bounds $u-G_0$ by $C\cD(y)^{-2}$ in $C^2$, so it remains to estimate the contribution of the omitted profiles. For this, observe that a segment inside the chart from the region between the retained sheets to the $k$-th further sheet crosses $k$ consecutive gaps. Thus, by \eqref{eq:ww-sheet-separation}, its length is at least $k(\sqrt2\log R-C)$. Enlarging the region by a fixed unit neighborhood reduces this lower bound by at most a universal constant. The exponential decay of the profiles and their derivatives therefore bounds the sum of the omitted tails, together with their first two derivatives, by
\[
 C\sum_{k\geq1}e^{-\sqrt2 k(\sqrt2\log R-C)}
 +CQe^{-cR}\leq CR^{-2}.
\]
(Here we used the bounded derivatives of the signed distances and $Q\leq CR$; the last term accounts for sheets at distance at least $cR$.) If no upper neighbor is present, only the lower-side sum occurs. The signs in \eqref{eq:normal-two-profiles} follow because $u>0$ between these two sheets, whose increasing-$u$ normals are oppositely oriented.

If no upper neighbor is present, \eqref{eq:normal-two-profiles} immediately gives \eqref{eq:normal-tangential-bounds}, since $g(t)$ is constant on $\Gamma^t$. We may therefore assume that an upper neighbor is present. Suppose first that $b_+(y)$ is an equidistant endpoint in \eqref{eq:ww-endpoints-definition}. The separation estimate in Theorem~\ref{lem:wang-wei-profiles} gives
\begin{equation}\label{eq:normal-endpoint-angle}
 e^{-\sqrt2b_+(y)}
 \bigl(2b_+(y)-\sqrt2\log \cD(y)+C\bigr)
 \leq C\cD(y)^{-1}.
\end{equation}
At $t=b_+(y)-s$, $0\leq s\leq b_+(y)$, the normal $\nu_+$ to the neighboring sheet, evaluated at the nearest point to $F(y,t)$, differs from its value at $t=b_+(y)$ by at most $Cs\cD(y)^{-1}$, while $r\geq b_+(y)+s/2$. Combining the normal comparison \eqref{eq:ww-endpoint-angle} with \eqref{eq:normal-endpoint-angle}, we obtain
\begin{equation}\label{eq:normal-neighbor-derivative}
 \begin{aligned}
 (g'(r)+|g''(r)|)|\nabla_{\Gamma^t}r|
 &\leq C\cD(y)^{-1}e^{-\sqrt2b_+(y)-s/\sqrt2}
       \bigl(2b_+(y)-\sqrt2\log \cD(y)+C+s\bigr) \leq C\cD(y)^{-2}.
 \end{aligned}
\end{equation}
Estimate \eqref{eq:normal-neighbor-derivative} also holds for $b_+(y)\leq t<b_+(y)+1$: the distance $r$ changes by at most $1$, and the neighboring normal $\nu_+=-\nabla r$ changes by at most $C\cD(y)^{-1}$, by \eqref{eq:ww-graph-bounds}. Moreover $g'(r)\leq C\cD(y)^{-1}$ and $|\nabla_{\Gamma^t}^2r|\leq C\cD(y)^{-1}$, by \eqref{eq:ww-graph-bounds} for the two sheets. By the chain rule
\begin{equation}\label{eq:normal-distance-profile-hessian}
 \nabla_{\Gamma^t}^2g(r)
 =g''(r)\nabla_{\Gamma^t}r\otimes\nabla_{\Gamma^t}r
       +g'(r)\nabla_{\Gamma^t}^2r,
\end{equation}
so, since $g(t)$ is constant on $\Gamma^t$, \eqref{eq:normal-two-profiles}, \eqref{eq:normal-neighbor-derivative}, and \eqref{eq:normal-distance-profile-hessian} imply \eqref{eq:normal-tangential-bounds} on this half of the interval. If instead $b_+(y)=T(y)$ is reached before equidistance, then $r(y,T(y))\geq T(y)$. If $r(y,t_0)<T(y)$ for some $t_0\in[0,T(y)]$, its $1$-Lipschitz property gives $r(y,t)\leq r(y,t_0)+|t-t_0|<2T(y)=6\log \cD(y)$ throughout $[0,T(y)]$. Then \cite[Lemma~3.4]{WW2019} and \eqref{eq:normal-two-profiles} give $\partial_t r\leq0$, so $r(y,T(y))\leq r(y,t_0)<T(y)$, a contradiction. Thus $r(y,t)\geq T(y)$ on this segment. Using $|\partial_t r|\leq1$ on the additional unit strip, the neighboring tail satisfies
\begin{equation*}
 \begin{split}
 r(y,t)&\geq T(y)-(t-T(y))_+\geq T(y)-1, \\
 \sum_{k=0}^2|D^k(g(r)-1)|&
       \leq Ce^{-\sqrt2 r(y,t)}
       \leq C\cD(y)^{-3\sqrt2}\leq C\cD(y)^{-2},
 \qquad 0\leq t<T(y)+1.
 \end{split}
\end{equation*}
(Here we used $|\nabla r|=1$, $|D^2r|\leq C\cD(y)^{-1}$, and $T(y)=3\log\cD(y)$.) This proves \eqref{eq:normal-tangential-bounds} also in the truncated case, including its exterior unit strip. The same argument using the lower neighboring sheet proves \eqref{eq:normal-tangential-bounds} for $t\leq0$. For $(y,t)\in\mathcal N$, every other sheet is at distance at least $|t|-2$. Summing the exponentially decaying first two profile derivatives proves \eqref{eq:normal-cell-derivative-tail}.

For \eqref{eq:normal-mean-curvature-gradient}, the bounds in \eqref{eq:ww-graph-bounds} give
\begin{equation}\label{eq:normal-mean-curvature-inputs}
 \|H_\Gamma\|_{C^{0,1/2}}\leq C\cD(y)^{-2},\qquad
 \|\mathrm{II}_\Gamma\|_{C^{0,1}}\leq C\cD(y)^{-1},
\end{equation}
with the norms taken in patches of fixed size. To estimate $DH_\Gamma$, we interpolate the H\"older bound for $H_\Gamma$ with a uniform bound for $D^2H_\Gamma$. To obtain the latter, note that \eqref{eq:ww-core-profile} gives $|\nabla u|\geq g'(0)-C\cD(y)^{-2}\geq c>0$ at every zero. Together with interior regularity for \eqref{eq:allen-cahn}, this gives uniform $C^4$ graph bounds by the implicit function theorem, and hence $|D^2H_\Gamma|\leq C$ in graph coordinates. Combining this bound with \eqref{eq:normal-mean-curvature-inputs} and applying Taylor's formula along a coordinate line, we obtain
\begin{equation}\label{eq:normal-mean-curvature-interpolation}
 |DH_\Gamma|\leq C\cD(y)^{-2}h^{-1/2}+Ch
 \leq C\cD(y)^{-4/3},\qquad h=\cD(y)^{-4/3}.
\end{equation}
The second fundamental form of the parallel surface is $ \mathrm{II}_{\Gamma^t} =\mathrm{II}_\Gamma(\mathrm{Id}-t\mathrm{II}_\Gamma)^{-1}. $ Differentiating its trace and using \eqref{eq:normal-mean-curvature-inputs} and \eqref{eq:normal-mean-curvature-interpolation} gives
\[
 \begin{aligned}
 |\nabla_{\Gamma^t}H_{\Gamma^t}|
 &\leq C|\nabla_\Gamma H_\Gamma|
       +C|t|\,|\mathrm{II}_\Gamma|
                       |\nabla_\Gamma\mathrm{II}_\Gamma| \leq C\cD(y)^{-4/3}+C\cD(y)^{-2}\log \cD(y)
 \leq C\cD(y)^{-4/3},
 \end{aligned}
\]
which proves \eqref{eq:normal-mean-curvature-gradient}.
\end{proof}

Finally, in the next lemma we compare $|\nabla u|$ and $\partial_t u$ across normal rays, and we show that $(\partial_t u)_+$ is small after the endpoints; in an idealized situation, see the figure below, we cross maxima/minima of $u$ and thus $\partial_t u$ changes sign.

\begin{figure}[ht]
\centering
\begin{tikzpicture}[x=1cm,y=0.85cm,>=stealth]
  \draw[thin,gray!30] (-4.5,1)--(4.5,1);
  \draw[thin,gray!30] (-4.5,-1)--(4.5,-1);

  \draw[->] (-4.65,0)--(4.85,0) node[right] {$t$};
  \draw[->] (0,-1.65)--(0,1.75) node[above] {$u$};
  \draw (-0.05,1)--(0.05,1);
  \draw (-0.05,-1)--(0.05,-1);
  \node[left,fill=white,inner sep=1pt] at (-0.08,1) {$1$};
  \node[left,fill=white,inner sep=1pt] at (-0.08,-1) {$-1$};

  \draw[dashed,gray] (-1.5,-1.35)--(-1.5,1.4);
  \draw[dashed,gray] (1.5,-1.35)--(1.5,1.4);
  \draw[thick]
    (-4.5,0.92)
    .. controls (-3.5,0.92) and (-3.0,-0.92) .. (-1.5,-0.92)
    .. controls (-0.5,-0.92) and (0.5,0.92) .. (1.5,0.92)
    .. controls (3.0,0.92) and (3.5,-0.92) .. (4.5,-0.92);
  \fill (-1.5,-0.92) circle (0.7pt);
  \fill (1.5,0.92) circle (0.7pt);

  \node[above left,fill=white,inner sep=1pt] at (-1.5,0) {$b_-$};
  \node[below right,fill=white,inner sep=1pt] at (1.5,0) {$b_+$};
  \node[below right] at (0,0) {$0$};
  \node at (-3,-1.75) {$\partial_tu<0$};
  \node at (0,-1.75) {$\partial_tu>0$};
  \node at (3,-1.75) {$\partial_tu<0$};
\end{tikzpicture}
\caption{\small Schematic periodic one-dimensional transition.}
\label{fig:normal-midpoint-sign}
\end{figure}
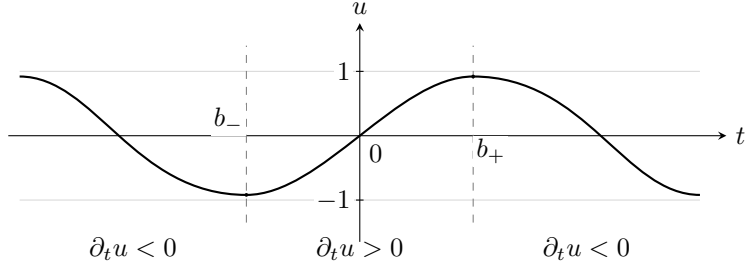

\begin{lem}[Comparison with the positive normal derivative]
\label{lem:normal-gradient-comparison}
In the setting fixed at the beginning of this section,
\begin{equation}\label{eq:normal-gradient-comparison}
 0\leq|\nabla u|-(\partial_tu)_+\leq C\cD(y)^{-2}
 \quad\text{for}\quad b_-(y)\leq t\leq b_+(y).
\end{equation}
We also have the following bound on the unit strips beyond the endpoints:
\begin{equation}\label{eq:normal-endpoint-positive}
 (\partial_tu)_+\leq C\cD(y)^{-2}
 \quad\text{if }\quad
 b_+(y)\leq t<b_+(y)+1
 \quad\text{or}\quad b_-(y)-1<t\leq b_-(y).
\end{equation}
\end{lem}

\begin{proof}
Use the profile representation \eqref{eq:normal-two-profiles} from the proof of Lemma~\ref{lem:layerwise-stability}. If no upper neighbor is present in the local chart, it gives $\nabla u=g'(t)\nu(y)+O(\cD(y)^{-2})$. Any zero competing before truncation lies within distance $2T(y)$ of $y$, hence belongs to the graphs in \eqref{eq:ww-graph-bounds}. Since there is no upper neighboring graph, $b_+(y)=T(y)$. Thus $g'(t)\leq C\cD(y)^{-3\sqrt2}$ on the exterior unit strip, and both conclusions follow on the upper half.

We may therefore assume that an upper neighbor is present. On $0\leq t\leq b_+(y)$, we have $r\geq t$ and $g'$ decreases on $[0,\infty)$, so
\begin{equation}\label{eq:normal-derivative-inside}
 \partial_tu=g'(t)-g'(r)\nu(y)\cdot\nu_+
       +O(\cD(y)^{-2})\geq-C\cD(y)^{-2}.
\end{equation}
Combining \eqref{eq:normal-derivative-inside} with \eqref{eq:normal-tangential-bounds} and $|\nabla u|\leq|\partial_tu|+|\nabla_{\Gamma^t}u|$ proves \eqref{eq:normal-gradient-comparison} on the upper half. Beyond an equidistant endpoint, on $b_+(y)\leq t<b_+(y)+1$, we instead have $r\leq t$, and hence
\begin{equation*}
 \partial_tu
 \leq g'(r)(1-\nu(y)\cdot\nu_+)+C\cD(y)^{-2}
 \leq C\cD(y)^{-2},
\end{equation*}
by \eqref{eq:normal-neighbor-derivative} and near parallelism. At a truncated endpoint, the derivatives of the profiles of the starting and neighboring sheets are bounded by $Ce^{-\sqrt2T(y)}\leq C\cD(y)^{-2}$, including on the exterior unit strip. The same argument using the lower neighboring sheet gives the bounds for $t\leq0$, completing \eqref{eq:normal-gradient-comparison} and \eqref{eq:normal-endpoint-positive}.
\end{proof}

\subsection{The stability inequality}

We use the coordinates of Definition~\ref{defn:normal-coordinates-endpoints} and the estimates above. We first record two crude area bounds.
\begin{lem}[Local and cubic area bounds]\label{lem:area-bounds}
In the setting of Definition~\ref{defn:bounded-cluster}, there is a universal constant $C>0$ such that
\begin{equation}\label{eq:clean-local-area}
 \mathcal H^2(\Sigma_0\cap B_1(x))\leq C 
 \quad\text{if }\operatorname{dist}(x,\mathcal X)\geq2,
\end{equation}
and
\begin{equation}\label{eq:cluster-cubic-area}
 \mathcal H^2\bigl(\Sigma\cap\{\cD<t\}\bigr)
 \leq CNt^3\quad\text{for all}\quad t\geq L_0.
\end{equation}
\end{lem}

\begin{proof}
At each $y\in\Sigma_0\setminus\mathcal X$, $\Sigma_0\cap B_{1/2}(y)$ is a graph with uniformly bounded slope. This gives a uniform bound for the area in $B_{1/2}(y)$. If $\operatorname{dist}(x,\mathcal X)\geq2$, a finite cover gives \eqref{eq:clean-local-area}.

Fix $t\geq L_0$ and choose a maximal one-separated family in $\Sigma\cap\{\cD<t\}$. Its disjoint half-unit balls lie in $\mathcal S_{t+1/2}$, which \eqref{eq:cluster-cover-mass} covers by $N$ balls of radius $t+3/2$. Volume comparison gives at most $CNt^3$ centers. Their unit balls cover $\Sigma\cap\{\cD<t\}$, and each has zero-set area at most $C$ by \eqref{eq:clean-local-area}: \eqref{eq:cluster-full-distance} gives the centers distance $\cD>L_0\geq2$ from the full bad set $\mathcal X$. Summing these estimates proves \eqref{eq:cluster-cubic-area}.
\end{proof}

We reduce stability to a condition on $\Sigma$, while retaining the curvature contribution of the bad set (recall \eqref{eq:bad-set}).

\begin{prop}[A geometric form of stability]\label{prop:bad-ball-stability}
Consider the setting in Definition~\ref{defn:bounded-cluster}. After increasing $R_*$ if necessary, assume that $L_0\geq4R_*$ and $\Lambda\geq32L_0$. Let $\varphi:\Sigma_0\to[0,1]$ satisfy $\varphi|_\Sigma\in C^{0,1}(\Sigma)$ and
\[
 \varphi=1\quad\text{on }\Sigma_0\cap\mathcal S_{4L_0},\qquad
 \varphi=0\quad\text{on }\Sigma_0\cap\{\cD\geq\Lambda/4\}.
\]
Then, there is a universal $C$ such that
\begin{equation}\label{eq:bad-ball-stability-summable}
 \begin{aligned}
 \frac1{\sigma_0}\int_{\mathcal S_2}\mathcal A^2|\nabla u|^2
 +\int_{\Sigma\setminus\mathcal S_{4L_0}}|\mathrm{II}_\Sigma|^2\varphi^2
 &\leq\frac32\int_\Sigma|\nabla_\Sigma\varphi|^2 +C\int_{\Sigma\setminus\mathcal S_{2L_0}}\cD^{-13/4}\varphi^2
       +CNL_0^{-1/4}.
 \end{aligned}
\end{equation}
\end{prop}
\begin{proof}
Let us briefly explain the idea of the proof. Away from the bad set, extend $\varphi$ along the normal rays and cut it off near the ends of each normal interval. The usual Sternberg--Zumbrun computation tests stability with a function of the form $\xi=\varphi |\nabla u|$. The idea is to use $\widetilde\xi=\widetilde\varphi(\partial_tu)_+$ instead: by \eqref{eq:normal-gradient-comparison}, the positive normal derivative approximates $|\nabla u|$ inside the interval, while \eqref{eq:normal-endpoint-positive} makes it small on the exterior unit strips where the normal cutoffs vary. In those strips, $\widetilde\varphi$ is obtained by cutting off $\varphi$, and this smallness allows us to control the resulting cutoff errors.

Now, importantly, $\widetilde\xi$ need not actually define a single-valued ambient test function, since (even after introducing the cutoffs) there might be an overlap in the exterior unit strips. We therefore work on the ``parametrizing manifold'' $\mathcal N$ from \eqref{eq:normal-param-manifold} with the pulled-back Euclidean metric instead. Stability provides a positive solution of the linearized equation, whose pullback allows us to use a weighted integration-by-parts identity on $\mathcal N$. We apply this identity to $\widetilde\xi$ on $\mathcal N$ and to $|\nabla u|$ in the ambient space, with complementary cutoff weights. In the latter application, the pointwise Sternberg--Zumbrun identity yields the curvature term near the bad set. Comparing the resulting formulas allows us to retain this positive contribution.

We now give the precise argument. Let us first establish and recall some conventions and setting. Translate $z_\Lambda$ to the origin. Recall that derivatives and integrals on $\mathcal N$ will be computed with the Euclidean metric pulled back by $F(y,t)=y+t\nu(y)$. Similarly, ambient functions are composed with $F$ in the coordinates $(y,t)$. In particular $\partial_tu(y,t)=\partial_t(u\circ F)(y, t) = \nabla u(F(y, t))\cdot \nu(y)$ and
\[
 \int_{\mathcal N}f
 =\int_\Sigma\int_{b_-(y)-1}^{b_+(y)+1}
 f(y,t)\det(\mathrm{Id}-t\mathrm{II}_\Sigma(y))\,dt\,d\mathcal H^2(y).
\]
Each starting zero is integrated once, whether or not the extended intervals have overlapping images. Choose $\chi:\mathbb R\to[0,1]$ smooth, with $\chi\equiv 0$ on $(-\infty,0]$ and $\chi\equiv 1$ on $[1/2,\infty)$. Set
\begin{equation*}
 \begin{split}
 \widetilde\varphi(y,t)&=\varphi(y)\big[\chi(\cD(y)-L_0)
                   \chi(b_+(y)+1-t)\chi(t-b_-(y)+1)\big],\\
 \widetilde\xi(y,t)&=\widetilde\varphi(y,t)(\partial_tu(y,t))_+,
 \end{split}
\end{equation*}
so that $\widetilde \varphi$ vanishes at $\cD(y)=L_0$ and at the normal endpoints $t=b_-(y)-1$ and $t=b_+(y)+1$, and it differs from $\varphi$ only at unit distance from these boundaries, where $\nabla \chi$ is active. In particular, $\widetilde \varphi=0$ for $\cD\geq\Lambda/4$, $\widetilde\xi\in H_0^1(\mathcal N)$, and $F(\operatorname{supp}\widetilde\xi)\Subset B_{R_\Lambda+\Lambda/3}$.

Now, observe that we want to find the term $\int_{\mathcal S_2}\mathcal A^2|\nabla u|^2$ in \eqref{eq:bad-ball-stability-summable}, which comes from the usual Sternberg--Zumbrun computation with $\xi=\varphi|\nabla u|$ (and not $\widetilde \xi$). We will nevertheless be able to relate the two.

\medskip\noindent\textbf{Step 1.} There holds
\begin{equation}\label{eq:normal-joining-energy}
 \int_{\mathcal S_2}\mathcal A^2|\nabla u|^2
 \leq\int_{\mathcal N}(1-\chi_0)\widetilde\xi(-\Delta+W''(u))\widetilde\xi
       +CNL_0^{-1/4},
\end{equation}
where $\chi_0$ is a smooth cutoff, with $\chi_0\equiv 1$ on $\mathcal S_{5L_0/2}$, $\chi_0\equiv 0$ outside $\mathcal S_{3L_0}$ (thus ${\rm supp}\,\chi_0\subset B_{R_\Lambda+\Lambda/3}$), and satisfies
\begin{equation}\label{eq:normal-joining-cutoff}
 0\leq\chi_0\leq1,\qquad
 |\nabla\chi_0|\leq C/L_0,\qquad
 |\Delta\chi_0|\leq C/L_0^2.
\end{equation}
Since $\chi_0\equiv1$ on $\mathcal S_2$, it suffices to bound $\int\chi_0\mathcal A^2|\nabla u|^2$ by the right term in \eqref{eq:normal-joining-energy}. Differentiating the Allen--Cahn equation gives, pointwise on $\{\nabla u\ne0\}$,
\begin{equation}\label{eq:normal-sz-pointwise}
\begin{aligned}
 |\nabla u|(-\Delta+W''(u))|\nabla u|
 &=-\frac12\Delta|\nabla u|^2+|\nabla|\nabla u||^2
                                      +W''(u)|\nabla u|^2 =-|D^2u|^2+|\nabla|\nabla u||^2
 =-\mathcal A^2|\nabla u|^2.
\end{aligned}
\end{equation}
This motivates using the following weighted ground-state identity, arising from stability: For test functions $f$ and $\eta$, with $\eta\geq0$ smooth, and either $\eta$ compactly supported or $f$ with zero trace, there holds
\begin{equation}\label{eq:normal-weighted-identity}
 \int\eta f(-\Delta+W''(u))f
 =\int\eta h^2|\nabla(f/h)|^2
  -\int f^2\bigl(\tfrac12\Delta\eta+
                      \nabla\eta\cdot\nabla\log h\bigr),
\end{equation}
where $h$ solves the linearized equation
\begin{equation}\label{eq:normal-linearized-h}
\left\{
\begin{array}{rcll}
-\Delta h+W''(u)h&=&0 & \text{in }B_{R_\Lambda+\Lambda},\\
h&=&1 & \text{on }\partial B_{R_\Lambda+\Lambda}.
\end{array}
\right.
\end{equation}
Stability and strict domain monotonicity of the first Dirichlet eigenvalue give $\lambda_1(-\Delta+W''(u),B_{R_\Lambda+\Lambda})>0$. The Dirichlet problem is therefore uniquely solvable; integrating by parts gives \eqref{eq:normal-weighted-identity}. The maximum principle gives $h>0$, and Harnack and interior gradient estimates at unit scale give
\begin{equation}\label{eq:normal-positive-solution}
 |\nabla\log h|\leq C
 \quad\text{on }B_{R_\Lambda+\Lambda/2}.
\end{equation}
Applying \eqref{eq:normal-weighted-identity} with $f=|\nabla u|$, $\eta=\chi_0$, and using \eqref{eq:normal-sz-pointwise}, we find then
\begin{equation}\label{eq:normal-ambient-square-identity}
\begin{aligned}
\int\chi_0\mathcal A^2|\nabla u|^2-\int_{\mathbb R^3}|\nabla u|^2
      \bigl(\tfrac12\Delta\chi_0+\nabla\chi_0\cdot\nabla\log h\bigr)=-\int\chi_0h^2\left|\nabla\frac{|\nabla u|}{h}\right|^2\le 0.
\end{aligned}
\end{equation}
We now apply \eqref{eq:normal-weighted-identity} with $f=\widetilde\xi$, $\eta=1-\chi_0$, and on $\mathcal N$ instead. To be precise: $F$ is a local isometry on $\mathcal N$, thus $h\circ F$ satisfies the same linearized equation. We can therefore perform on $\mathcal N$ the same integration by parts that leads to \eqref{eq:normal-weighted-identity} from \eqref{eq:normal-linearized-h}. Likewise, there are no boundary terms, thanks to $\widetilde\xi$ vanishing at $t=b_--1$ and $t=b_++1$. We then find
\begin{equation}\label{eq:normal-square-identity0}
\begin{aligned}
 \int_{\mathcal N}(1-\chi_0)\widetilde\xi(-\Delta+W''(u))\widetilde\xi-\int_{\mathcal N}\widetilde\xi^2
      \bigl(\tfrac12\Delta\chi_0+\nabla\chi_0\cdot\nabla\log h\bigr)
 &=\int_{\mathcal N}(1-\chi_0)h^2
                 |\nabla(\widetilde\xi/h)|^2\ge 0. 
\end{aligned}
\end{equation}
Subtracting \eqref{eq:normal-square-identity0} from \eqref{eq:normal-ambient-square-identity}, and setting $\Psi_0:=\tfrac12\Delta\chi_0+\nabla\chi_0\cdot\nabla\log h$, we get
\begin{equation}\label{eq:normal-square-identity}
\begin{aligned}
 \int\chi_0\mathcal A^2|\nabla u|^2\leq \int_{\mathcal N}(1-\chi_0)\widetilde\xi(-\Delta+W''(u))\widetilde\xi+\int_{\mathbb R^3}|\nabla u|^2
      \Psi_0-\int_{\mathcal N}\widetilde\xi^2
      \Psi_0. 
\end{aligned}
\end{equation}
To reach \eqref{eq:normal-joining-energy} it only remains to bound the difference between the last two terms by $CNL_0^{-1/4}$. We compute it by splitting into three regions:
\begin{equation}\label{eq:normal-weighted-joining}
\begin{aligned}
 \int_{\mathbb R^3}& |\nabla u|^2
 \Psi_0
 -\int_{\mathcal N}\widetilde\xi^2
 \Psi_0 =\int_{\mathcal N\cap\{b_-<t<b_+\}}
       \bigl(|\nabla u|^2-\widetilde\xi^2\bigr)
\Psi_0+\int_{\mathbb R^3\setminus F(\mathcal N\cap\{b_-<t<b_+\})}
       |\nabla u|^2
\Psi_0 -\int_{\mathcal N\cap(\{t\leq b_-\}\cup\{t\geq b_+\})}
       \widetilde\xi^2
\Psi_0\\
 &\leq\frac C{L_0}\int_{\Sigma\cap\{2L_0<\cD<4L_0\}}
          \int_{b_-}^{b_+}\bigl(|\nabla u|^2-\widetilde\xi^2\bigr)\,dt +\frac C{L_0}
       \int_{\mathcal S_{3L_0}\setminus
         (\mathcal S_{5L_0/2}\cup F(\mathcal N\cap\{b_-<t<b_+\}))}
                       |\nabla u|^2\\
 &\qquad+\frac C{L_0}\int_{\Sigma\cap\{2L_0<\cD<4L_0\}}
          \bigg(\int_{b_--1}^{b_-}\widetilde\xi^2\,dt
                +\int_{b_+}^{b_++1}\widetilde\xi^2\,dt\bigg).
\end{aligned}
\end{equation}
Here we used the support and derivative bounds for $\chi_0$ in \eqref{eq:normal-joining-cutoff}, the bound for $\nabla\log h$ in \eqref{eq:normal-positive-solution}, and the bounded Jacobian. Moreover, every interval meeting the support of the derivatives of $\chi_0$ starts at $2L_0<\cD(y)<4L_0$, where $\varphi=1$. On its unextended part $\widetilde\varphi=1$ and $\widetilde\xi=(\partial_tu)_+$, so the differences of squares in the expression are nonnegative.

Finally, we claim the following bounds:
\begin{align}
 \frac1{L_0}\int_{\Sigma\cap\{2L_0<\cD<4L_0\}}
       \int_{b_-}^{b_+}\bigl(|\nabla u|^2-\widetilde\xi^2\bigr)\,dt
 &\leq CNL_0^{-2}\log L_0,
 \label{eq:normal-joining-error}\\
  \frac1{L_0}
       \int_{\mathcal S_{3L_0}\setminus
         (\mathcal S_{5L_0/2}\cup F(\mathcal N\cap\{b_-<t<b_+\}))}
                       |\nabla u|^2
 &\leq CNL_0^{-4},
 \label{eq:normal-outside-interval-error}\\
 \frac1{L_0}\int_{\Sigma\cap\{2L_0<\cD<4L_0\}}
       \bigg(\int_{b_--1}^{b_-}\widetilde\xi^2\,dt
             +\int_{b_+}^{b_++1}\widetilde\xi^2\,dt\bigg)
 &\leq CNL_0^{-2}.
 \label{eq:normal-joining-strip-error}
\end{align}

For \eqref{eq:normal-joining-error}, on the unextended intervals $\varphi=\widetilde\varphi=1$, so $|\nabla u|^2-\widetilde\xi^2 =|\nabla_{\Gamma^t}u|^2+(\partial_tu)_-^2\leq C\cD(y)^{-4}$ by \eqref{eq:normal-tangential-bounds} and \eqref{eq:normal-gradient-comparison}; we are using that if $\partial_tu<0$, \eqref{eq:normal-gradient-comparison} bounds $|\nabla u|$ itself. Integrating over length at most $6\log \cD(y)$ and using Lemma~\ref{lem:area-bounds} proves the first bound.

To prove \eqref{eq:normal-joining-strip-error}, we combine the same area estimate with $\widetilde\xi^2\leq C\cD(y)^{-4}$ on the exterior unit strips (recall \eqref{eq:normal-endpoint-positive}).

Finally, for \eqref{eq:normal-outside-interval-error}, almost every point in the domain of integration satisfies $\operatorname{dist}(x,\Sigma_0)\geq3\log L_0$. Indeed, almost every closer point has a unique nearest zero $y$ with $2L_0<\cD(y)<4L_0$ and $|t|<3\log L_0<T(y)$; the same zero remains uniquely nearest along the segment, so $b_-(y)<t<b_+(y)$, contrary to the exclusion of those images. The proof of Lemma~\ref{lem:uniform-regularity-phase-decay}, using $|u|(1+|u|)\geq3/2$ sufficiently far from $\Sigma_0$, gives $|\nabla u|\leq Ce^{-\operatorname{dist}(x,\Sigma_0)}$. Thus the last term is bounded by $CL_0^{-7}|\mathcal S_{3L_0}|\leq CNL_0^{-4}$, by \eqref{eq:cluster-cover-mass}.

Substituting \eqref{eq:normal-joining-error}-\eqref{eq:normal-outside-interval-error}-\eqref{eq:normal-joining-strip-error} into \eqref{eq:normal-square-identity}-\eqref{eq:normal-weighted-joining} proves \eqref{eq:normal-joining-energy}.
\par\medskip

We now compare the integral on the right of \eqref{eq:normal-joining-energy} with the geometric terms in \eqref{eq:bad-ball-stability-summable} and conclude:

\medskip\noindent\textbf{Step 2.} There holds
\begin{equation}\label{eq:normal-weighted-sheets}
 \begin{aligned}
 \int_{\mathcal N}(1-\chi_0)\widetilde\xi(-\Delta+W''(u))\widetilde\xi
 &\leq\frac{9\sigma_0}{8}\int_\Sigma|\nabla_\Sigma\varphi|^2 -\frac{3\sigma_0}{4}
       \int_{\Sigma\setminus\mathcal S_{4L_0}}
                   |\mathrm{II}_\Sigma|^2\varphi^2
       +C\int_{\Sigma\setminus\mathcal S_{2L_0}}\cD^{-13/4}\varphi^2.
 \end{aligned}
\end{equation}
Indeed, first integrating by parts on $\mathcal N$ using $\widetilde\xi=(\partial_tu)_+\widetilde\varphi\in H_0^1(\mathcal N)$ we get
\begin{multline}\label{eq:normal-product-integration}
 \int_{\mathcal N}(1-\chi_0)\widetilde\xi
                         (-\Delta+W''(u))\widetilde\xi
 =\int_{\mathcal N}(1-\chi_0)(\partial_tu)_+\widetilde\varphi^2
                         (-\Delta+W''(u))\partial_tu\\
 +\int_{\mathcal N}(1-\chi_0)(\partial_tu)_+^2
                                      |\nabla\widetilde\varphi|^2
 -\int_{\mathcal N}(\partial_tu)_+^2\widetilde\varphi\,
                          \nabla\chi_0\cdot\nabla\widetilde\varphi.
\end{multline}
Then, substituting \eqref{eq:normal-differentiated-equation} into \eqref{eq:normal-product-integration} gives
\begin{multline}\label{eq:normal-weighted-test}
 \int_{\mathcal N}(1-\chi_0)\widetilde\xi(-\Delta+W''(u))\widetilde\xi
 =\int_{\mathcal N}(1-\chi_0)(\partial_tu)_+^2
       (|\nabla\widetilde\varphi|^2-|\mathrm{II}_{\Gamma^t}|^2\widetilde\varphi^2)\\
 +\underbrace{\int_{\mathcal N}(1-\chi_0)(\partial_tu)_+\widetilde\varphi^2
       \bigl(2\mathrm{II}_{\Gamma^t}:\nabla_{\Gamma^t}^2u
            +\nabla_{\Gamma^t}H_{\Gamma^t}
                         \cdot\nabla_{\Gamma^t}u\bigr)}_{\mathrm{I_1}}
                         -\underbrace{\int_{\mathcal N}(\partial_tu)_+^2\widetilde\varphi\,
\nabla\chi_0\cdot\nabla\widetilde\varphi}_{\mathrm{I_2}}.
 \end{multline}
Using $ab\leq a^2/8+2b^2$ we have
\[
 \bigl(|\mathrm{II}_{\Gamma^t}|(\partial_tu)_+\bigr)
 \bigl(2|\nabla_{\Gamma^t}^2u|\bigr)
 \leq\tfrac18|\mathrm{II}_{\Gamma^t}|^2(\partial_tu)_+^2
      +8|\nabla_{\Gamma^t}^2u|^2,
\]
so that by \eqref{eq:normal-tangential-bounds} and \eqref{eq:normal-mean-curvature-gradient} we can bound
\begin{equation*}
 \begin{aligned}
 \mathrm{I_1}
 &\leq\tfrac18\int_{\mathcal N}(1-\chi_0)|\mathrm{II}_{\Gamma^t}|^2
                         (\partial_tu)_+^2\widetilde\varphi^2
       +C\int_{\Sigma\setminus\mathcal S_{2L_0}}
                       (\cD^{-4}\log \cD+\cD^{-10/3}\log \cD)\,\varphi^2\\
 &\leq\tfrac18\int_{\mathcal N}(1-\chi_0)|\mathrm{II}_{\Gamma^t}|^2
                         (\partial_tu)_+^2\widetilde\varphi^2
       +C\int_{\Sigma\setminus\mathcal S_{2L_0}}\cD^{-13/4}\varphi^2.
 \end{aligned}
\end{equation*}
Here we used $b_+-b_-+2\leq6\log \cD(y)+2$ and $(\partial_tu)_+\leq C$.

We next claim that
\begin{equation*}
 |\mathrm{I_2}|
 \leq C\int_{\Sigma\setminus\mathcal S_{2L_0}}\cD^{-5}\varphi^2.
\end{equation*}
On $\{\nabla\chi_0\ne0\}$ we have $2L_0<\cD(y)<4L_0$ and $\varphi=\chi(\cD(y)-L_0)=1$. Then $\nabla\widetilde\varphi=0$ for $b_-<t<b_+$, so the integrand defining $\mathrm{I_2}$ can be nonzero only for
\[
 b_--1<t<b_-\qquad\text{or}\qquad b_+<t<b_++1.
\]
There, \eqref{eq:ww-endpoint-bounds} and \eqref{eq:normal-endpoint-positive} give
\[
 \big|(\partial_tu)_+^2\widetilde\varphi\,
                   \nabla\chi_0\cdot\nabla\widetilde\varphi\big|
 \leq C\cD(y)^{-4}\varphi^2\,\cD(y)^{-1}.
\]
Integrating proves the claim.

Combining these bounds, \eqref{eq:normal-weighted-test} becomes
\begin{equation}\label{eq:normal-after-errors}
\begin{aligned}
 \int_{\mathcal N}(1-\chi_0)\widetilde\xi(-\Delta+W''(u))\widetilde\xi
 &\leq\int_{\mathcal N}(1-\chi_0)(\partial_tu)_+^2
     \bigl(|\nabla\widetilde\varphi|^2
           -\tfrac78|\mathrm{II}_{\Gamma^t}|^2\widetilde\varphi^2\bigr)
 +C\int_{\Sigma\setminus\mathcal S_{2L_0}}\cD^{-13/4}\varphi^2.
\end{aligned}
\end{equation}

It remains to pass from $\mathcal N$ to $\Sigma$ in the first integral on the right, by integrating out the normal directions. We claim that, for $L_0$ sufficiently large,
\begin{equation*}
 \begin{aligned}
 \int_{\mathcal N}(1-\chi_0)(\partial_tu)_+^2|\nabla\widetilde\varphi|^2
 &\leq\frac{9\sigma_0}{8}\int_\Sigma|\nabla_\Sigma\varphi|^2
       +C\int_{\Sigma\setminus\mathcal S_{2L_0}}\cD^{-4}\varphi^2,\\
 \int_{\mathcal N}(1-\chi_0)|\mathrm{II}_{\Gamma^t}|^2
                         (\partial_tu)_+^2\widetilde\varphi^2
 &\geq\frac{7\sigma_0}{8}
       \int_{\Sigma\setminus\mathcal S_{4L_0}}
                         |\mathrm{II}_\Sigma|^2\varphi^2.
 \end{aligned}
\end{equation*}
For both inequalities, we use \eqref{eq:ww-core-profile} and \eqref{eq:normal-cell-derivative-tail}, which give
\begin{equation}\label{eq:normal-squared-derivative-integrals}
 \int_{-\frac12\log\cD(y)}^{\frac12\log\cD(y)}
       (\partial_tu)^2\,dt=\sigma_0+o(1),\qquad
 \int_{b_--1}^{b_++1}(\partial_tu)_+^2\,dt
       =\sigma_0+o(1),
\end{equation}
with $o(1)\to0$ uniformly for $\cD(y)\geq L_0$ as $L_0\to\infty$. On the smaller interval, $\partial_tu>0$.

For the first inequality, recall that $1-\chi_0=0$ if $\cD(y)\leq2L_0$. For $\cD(y)>2L_0$, we have $\widetilde\varphi=\varphi$ on $b_-<t<b_+$, while \eqref{eq:ww-endpoint-bounds} and \eqref{eq:normal-endpoint-positive} give
\begin{equation*}
 \left(\int_{b_--1}^{b_-}+\int_{b_+}^{b_++1}\right)
       (\partial_tu)_+^2|\nabla\widetilde\varphi|^2
                  \det(\mathrm{Id}-t\mathrm{II}_\Sigma)\,dt
 \leq C\cD(y)^{-4}\bigl(\varphi(y)^2+|\nabla_\Sigma\varphi(y)|^2\bigr).
\end{equation*}
Under $F$, tangential gradients are multiplied by $(\mathrm{Id}-t\mathrm{II}_\Sigma)^{-1} =\mathrm{Id}+O(|t|\cD(y)^{-1})$ and area by $\det(\mathrm{Id}-t\mathrm{II}_\Sigma)=1+O(|t|\cD(y)^{-1})$. Thus, \eqref{eq:normal-squared-derivative-integrals} proves the first inequality.

For the second inequality, restrict the integral to $(y,t)$ with $\cD(y)\geq4L_0$ and $|t|\leq\frac12\log\cD(y)$---we discard the rest of the domain by positivity. There $\chi_0=0$, $\widetilde\varphi=\varphi$, and $|\mathrm{II}_{\Gamma^t}|^2\det(\mathrm{Id}-t\mathrm{II}_\Sigma) \geq(1-C|t|\cD(y)^{-1})|\mathrm{II}_\Sigma|^2$. The first identity in \eqref{eq:normal-squared-derivative-integrals} therefore proves the second inequality.

Substituting these two inequalities into \eqref{eq:normal-after-errors} gives \eqref{eq:normal-weighted-sheets}. Together with \eqref{eq:normal-joining-energy} we conclude \eqref{eq:bad-ball-stability-summable}.
\end{proof}
We can readily absorb the two last ``error'' terms in \eqref{eq:bad-ball-stability-summable}. The power $-13/4<-3$ is crucial for this.

\begin{prop}[Geometric stability, without the error terms]\label{prop:cluster-stability}
Consider the setting in Definition~\ref{defn:bounded-cluster}. After increasing $R_*$ if necessary, assume that $L_0\geq4R_*$ and $\Lambda\geq32L_0$. Let $\varphi:\Sigma_0\to[0,1]$ satisfy $\varphi|_\Sigma\in C^{0,1}(\Sigma)$ and
\[
 \varphi=1\quad\text{on }\Sigma_0\cap\mathcal S_{4L_0},\qquad
 \varphi=0\quad\text{on }\Sigma_0\cap\{\cD\geq\Lambda/4\}.
\]
Then, for $L_0$ sufficiently large, we have
\begin{equation}\label{eq:block-classification-stability}
 \frac{\delta N}{2\cdot5^3\sigma_0}
 +\int_{\Sigma\setminus\mathcal S_{4L_0}}|\mathrm{II}_\Sigma|^2\varphi^2
 \leq\frac32\int_\Sigma|\nabla_\Sigma\varphi|^2.
\end{equation}
The lower threshold for $L_0$ is universal, independent of $u$, $N$, $R_\Lambda$ and $\Lambda$.
\end{prop}

\begin{proof}
Applying the cubic area bound in Lemma~\ref{lem:area-bounds} on the sets $2^kL_0\leq\cD<2^{k+1}L_0$, $k\geq1$, and using $0\leq\varphi\leq1$, we obtain
\begin{equation*}
 \int_{\Sigma\setminus\mathcal S_{2L_0}}\varphi^2 \cD^{-13/4}\,d\mathcal H^2\leq \int_{\Sigma\setminus\mathcal S_{2L_0}} \cD^{-13/4}\,d\mathcal H^2
 \leq CN\sum_{k=1}^{\infty}(2^kL_0)^{-1/4}
 \leq CNL_0^{-1/4}.
\end{equation*}
Choose $L_0$ sufficiently large that the last two terms in \eqref{eq:bad-ball-stability-summable} are bounded by $\delta N/(2\cdot5^3\sigma_0)$. Since \eqref{eq:cluster-cover-mass} gives $\sigma_0^{-1}\int_{\mathcal S_2}\mathcal A^2|\nabla u|^2 \geq\delta N/(5^3\sigma_0)$, these terms can be absorbed into the left-hand side, yielding \eqref{eq:block-classification-stability}.
\end{proof}
\section{Proof of Theorem~\ref{thm:conditional-classification}}
\label{sec:classification}
\label{sec:intrinsic-curvature}
We now complete the proof using the geometric stability inequality of Proposition~\ref{prop:cluster-stability}, intrinsic area estimates, and a logarithmic cutoff. We first take $L_0$ sufficiently large and construct a smooth neighborhood of the bad cluster at scale $L_0$, whose boundary meets $\Sigma$ transversely.

\begin{lem}[Large smooth neighborhood of the bad set]\label{lem:smooth-bad-region}
In the setting of Definition~\ref{defn:bounded-cluster}, there is a bounded open set $\mathcal B$ with smooth boundary such that:
\begin{enumerate}
\item There holds
\begin{equation}\label{eq:cluster-smooth-region}
 \mathcal S_{4L_0}\subset\mathcal B,\qquad
 \overline{\mathcal B}\subset\mathcal S_{4L_0+2}\Subset B_{R_\Lambda}(z_\Lambda).
\end{equation}
\item $\partial\mathcal B\subset\{3L_0<\cD<5L_0\}$ and $\partial\mathcal B$ is transverse to $\Sigma_0$. In particular, $\Sigma\cap\partial\mathcal B$ is a smooth compact curve, possibly empty or disconnected.
\end{enumerate}
\end{lem}

\begin{proof}
Mollify $\cD$ to obtain a smooth function $\widetilde{\cD}$ with $|\widetilde{\cD}-\cD|<1/4$. Since $\overline{\mathcal X_\Lambda}$ is bounded, $\widetilde{\cD}$ tends to infinity at infinity. By \eqref{eq:cluster-full-distance} and~\eqref{eq:cluster-curvature-bounds}, the zero set is smooth on $\{3L_0<\cD<5L_0\}\subset\{L_0<\cD<\Lambda/3\}$. Sard's theorem gives a common regular value $t_0\in(4L_0+1/2,4L_0+1)$ of $\widetilde{\cD}$ on $\mathbb R^3$ and on this part of $\Sigma_0$. Then $\mathcal B=\{\widetilde{\cD}<t_0\}$ is bounded with smooth boundary, and the approximation bound and~\eqref{eq:cluster-radial-bounds} give \eqref{eq:cluster-smooth-region}. Its boundary lies in $\{3L_0<\cD<5L_0\}$ and is transverse to $\Sigma_0$ by the choice of $t_0$.
\end{proof}

Fix this choice of $\mathcal B$. It plays the role of the bounded bad region in \cite[(10.13)--(10.17)]{CFFS2026}; the smaller set $\mathcal S_2\subset\mathcal B$ carries the positive curvature mass in \eqref{eq:cluster-cover-mass}.

We first introduce the intrinsic distance from the boundary of $\mathcal B$ and the area of its distance neighborhoods.

\begin{defn}[Intrinsic distance and area]\label{defn:intrinsic-distance-area}
For $y\in\Sigma\setminus\mathcal B$, let
\begin{equation*}
 d_{\mathcal B}(y)=
 \operatorname{dist}_{\Sigma\setminus\mathcal B}
       (y,\Sigma\cap\partial\mathcal B),
\end{equation*}
the infimum of lengths of paths in $\Sigma\setminus\mathcal B$ joining $y$ to $\Sigma\cap\partial\mathcal B$. It is $+\infty$ on components that do not meet this boundary, and we extend it by zero on $\Sigma\cap\mathcal B$.

We denote the area of an intrinsic neighborhood by
\begin{equation*}
 \Theta(r)=\mathcal H^2\bigl(\{y\in\Sigma:0<d_{\mathcal B}(y)<r\}\bigr).
\end{equation*}
\end{defn}

The next result relates intrinsic area with the curvature term in stability. We set $(r-d_{\mathcal B})_+=0$ where $d_{\mathcal B}=+\infty$.

\begin{lem}[Relating area and curvature]\label{lem:integrated-curvature}
For $2\leq r<\Lambda/16$ we have
\begin{equation}\label{eq:integrated-curvature}
 \Theta(r)\leq\frac14\int_{\Sigma\setminus\mathcal B}|\mathrm{II}_\Sigma|^2
 (r-d_{\mathcal B})_+^2\,d\mathcal H^2+Cr^2\Theta(2),
\end{equation}
for some universal constant $C$.
\end{lem}
\begin{proof}
An analogous area--curvature estimate appears in \cite[Lemma~10.16 and the proof of Proposition~10.17]{CFFS2026}. For completeness, we  give a self-contained proof using normal Jacobi fields in Appendix~\ref{sec:intrinsic-area-proof}.
\end{proof}

Finally, we can give the proof of the main result:

\begin{proof}[Proof of Theorem~\ref{thm:conditional-classification}]
Suppose that $u$ is not one-dimensional. Then $\mathcal X\ne\varnothing$ by Remark~\ref{rem:bad-set-nonempty}. Choose a sufficiently large universal $L_0$, and then $\Lambda$ sufficiently large depending only on $L_0$.

From now on, we put ourselves in the setting of Definition~\ref{defn:bounded-cluster}. Consider the smooth set $\mathcal B$ constructed in Lemma~\ref{lem:smooth-bad-region}, together with the intrinsic distance and area from Definition~\ref{defn:intrinsic-distance-area}. Note that Proposition~\ref{prop:cluster-stability} applies to every admissible $\varphi$, and that the lower thresholds for $L_0$ and $\Lambda$ depend only on universal constants, independently of the number $N$ of bad centers and the radius $R_\Lambda$ of the cluster.

\noindent\textbf{Step 1.} Quadratic area growth.
We first show that
\begin{equation}\label{eq:block-classification-area}
 \Theta(r)\leq Cr^2\Theta(2)\quad\text{for}\quad 2\leq r<\Lambda/16,\qquad \mbox{with}\quad
 \Theta(2)\leq CNL_0^3.
\end{equation}

Indeed, fix $r$ with $2\leq r<\Lambda/16$, and put $\varphi=1$ on $\Sigma_0\cap\mathcal B$ and $\varphi=(1-d_{\mathcal B}/r)_+$ on ${\Sigma\setminus\mathcal B}$, with zero on the remaining zero set. Its support outside $\mathcal B$ lies in
\begin{equation}\label{eq:block-classification-cutoff-support}
 \{4L_0\leq \cD\leq4L_0+2+r\}
 \Subset\{L_0<\cD<\Lambda/4\}.
\end{equation}
Thus $\varphi$ is admissible for Proposition~\ref{prop:cluster-stability}, which together with Lemma~\ref{lem:integrated-curvature} gives

\begin{equation*}
\begin{aligned}
 \frac{\Theta(r)}{r^2}
 &\leq\frac14\int_{\Sigma\setminus\mathcal B}|\mathrm{II}_\Sigma|^2
                    (1-d_{\mathcal B}/r)_+^2+C\Theta(2) \leq\frac38\int_\Sigma|\nabla_\Sigma\varphi|^2+C\Theta(2)
  =\frac38\frac{\Theta(r)}{r^2}+C\Theta(2).
\end{aligned}
\end{equation*}
Absorbing proves the first estimate in \eqref{eq:block-classification-area}. Finally, the bound for $\Theta(2)$ follows from \eqref{eq:cluster-cubic-area}, since $\{0<d_{\mathcal B}<2\}\subset\Sigma\cap\{\cD<4L_0+4\}$.

\noindent\textbf{Step 2.} Log-cutoff and contradiction.
We consider a second test function, with $\varphi=1$ on $\Sigma_0\cap\mathcal B$ and
\begin{equation}\label{eq:block-classification-log-cutoff}
 \varphi=
 \left(1-\frac{\log(\max\{2,d_{\mathcal B}\}/2)}
                      {\log(\Lambda/64)}\right)_+\qquad \mbox{on} \quad {\Sigma\setminus\mathcal B}.
\end{equation}
Put again $\varphi=0$ on the remaining zero set and where $d_{\mathcal B}=+\infty$. Its exterior support has $d_{\mathcal B}\leq\Lambda/32$, so \eqref{eq:block-classification-cutoff-support} with $r=\Lambda/32$ proves admissibility. On $2^j<d_{\mathcal B}<2^{j+1}$, \eqref{eq:block-classification-log-cutoff} gives $|\nabla_\Sigma\varphi|\leq2^{-j}/\log(\Lambda/64)$. The intervals meeting its gradient support have $j\geq1$ and $2^j<\Lambda/32$, hence $2^{j+1}<\Lambda/16$. Thus \eqref{eq:block-classification-area} applies on each interval and gives
\begin{equation*}
\begin{aligned}
 \int_\Sigma|\nabla_\Sigma\varphi|^2
 &\leq\frac1{\log^2(\Lambda/64)}
   \sum_{\substack{j\geq1\\2^j<\Lambda/32}}2^{-2j}\Theta(2^{j+1}) \leq\frac{C\Theta(2)}{\log\Lambda}
 \leq\frac{CNL_0^3}{\log\Lambda}.
\end{aligned}
\end{equation*}
Combining this with \eqref{eq:block-classification-stability} gives $ c\delta N\leq\frac{CNL_0^3}{\log\Lambda} $ which, for $\Lambda$ sufficiently large (depending only on $L_0$), gives a contradiction. Thus $u$ is one-dimensional. The only stable one-dimensional solutions are the constants $\pm 1$ and translates of the heteroclinic profile $g$, which together with $|u|<1$ gives $u(x)=g(e\cdot x-t_0)$.
\end{proof}

\appendix
\section{Proof of Lemma~\ref{lem:integrated-curvature}}\label{sec:intrinsic-area-proof}

\begin{proof}[Proof of Lemma~\ref{lem:integrated-curvature}]
We use the notation of Section~\ref{sec:intrinsic-curvature}.

By construction, $\Sigma\cap\partial\mathcal B$ is a smooth compact curve, possibly empty or disconnected. Every exterior path of length at most $r$ starting on ${\Sigma\cap\partial\mathcal B}$ lies in
\begin{equation}\label{eq:block-curvature-distance-region}
 \Sigma_0\cap\{4L_0\leq \cD\leq4L_0+2+r\}\Subset\Sigma,
 \qquad r\leq\Lambda/16.
\end{equation}
The lower bound follows from $\mathcal S_{4L_0}\subset\mathcal B$, and the upper bound from the fact that $\cD$ is $1$-Lipschitz. If $\Sigma\cap\partial\mathcal B=\varnothing$, then $d_{\mathcal B}=+\infty$ on $\Sigma\setminus\mathcal B$ and the conclusion \eqref{eq:integrated-curvature} is immediate.

Let $K$ denote the Gaussian curvature of $\Sigma$, and note that
\begin{equation}\label{eq:block-curvature-lower-bound}  
|\mathrm{II}_\Sigma|\leq\frac{|D^2u|}{|\nabla u|}\leq C,  \qquad K\geq-\tfrac12|\mathrm{II}_\Sigma|^2\geq-C^2
\end{equation}
for a fixed universal $C\geq1$. A minimizing path from $\Sigma\cap\partial\mathcal B$ cannot meet this boundary again. For each $q\in\Sigma\cap\partial\mathcal B$, consider the unit-speed geodesic starting at $q$, normal to this curve within $\Sigma$, and pointing into $\Sigma\setminus\mathcal B$. Write $\kappa(q)$ for its cut time, up to which it minimizes distance from the inner boundary, and $J(q,t)$ for the length Jacobian, with arclength on the inner boundary as the initial coordinate. Off the cut locus, which has surface measure zero, every point at finite distance below $\Lambda/16$ lies on a unique minimizing normal ray. The Jacobi equation and the area formula give
\begin{equation}\label{eq:block-curvature-rays}
 \begin{split}
 &J_{tt}(q,t)=-K(q,t)J(q,t),\qquad J(q,0)=1,
 \qquad J(q,t)>0\quad\text{for}\quad 0<t<\kappa(q), \\
 &\Theta(s)=\int_{\Sigma\cap\partial\mathcal B}\int_0^{\min\{s,\kappa(q)\}}
                  J(q,t)\,dt\,d\mathcal H^1(q)
 \quad\text{for}\quad 0<s<\Lambda/16.
 \end{split}
\end{equation}

For $0<t<\kappa(q)$, \eqref{eq:block-curvature-rays} and \eqref{eq:block-curvature-lower-bound} give
\begin{equation}\label{eq:block-curvature-comparison}
 \left(\frac{J_t}{J}\right)'
       +\left(\frac{J_t}{J}\right)^2=-K\leq C^2,\qquad
 \frac{J_t}{J}\leq C\coth(Ct).
\end{equation}
The second inequality in \eqref{eq:block-curvature-comparison} follows by comparison, since $J_t/J$ is finite at zero and $C\coth(Ct)\to+\infty$ there. In particular, the resulting bound is uniform on $[1/2,1]$ and independent of the inner boundary curvature.

We now choose a level in this interval whose length is also controlled. Since $d_{\mathcal B}$ is locally Lipschitz on its finite-distance components, with $|\nabla_\Sigma d_{\mathcal B}|=1$ almost everywhere at positive distance, coarea gives $b\in[1/2,1]$ such that $\mathcal H^1(\{d_{\mathcal B}=b\})\leq2\Theta(2)$. We may also choose this level to meet the cut locus in a set of length zero. Combining this choice with the comparison estimate gives
\begin{equation}\label{eq:block-curvature-slice}
 \begin{split}
 &J_t(q,b)_+\leq CJ(q,b)\quad\text{whenever}\quad \kappa(q)>b, \qquad \int_{\{q\in{\Sigma\cap\partial\mathcal B}:\kappa(q)>b\}}J(q,b)\,d\mathcal H^1(q)
 =\mathcal H^1(\{d_{\mathcal B}=b\})\leq2\Theta(2).
 \end{split}
\end{equation}

Fix $2\leq r<\Lambda/16$. For $b<t<\kappa(q)$, twice integrating the Jacobi equation gives
\[
 J(q,t)\leq J(q,b)+(t-b)J_t(q,b)_+
       +\int_b^t(t-s)K(q,s)_-J(q,s)\,ds.
\]
Integrate first over $b<t<\min\{r,\kappa(q)\}$ and then over the rays with $\kappa(q)>b$, taking a limit from below if the cut time is reached. To recover $\Theta(r)$, we also include the points at distance at most $b$, whose area is bounded by $\Theta(2)$. Adding this contribution and applying Fubini's theorem together with \eqref{eq:block-curvature-rays}--\eqref{eq:block-curvature-slice}, we obtain
\[
\begin{aligned}
 \Theta(r)&\leq Cr^2\Theta(2)
       +\frac12\int_{\Sigma\setminus\mathcal B} K_-\,(r-d_{\mathcal B})_+^2\,d\mathcal H^2 \leq Cr^2\Theta(2)+\frac14\int_{\Sigma\setminus\mathcal B}|\mathrm{II}_\Sigma|^2
                 (r-d_{\mathcal B})_+^2\,d\mathcal H^2,
\end{aligned}
\]
since $K_-\leq|\mathrm{II}_\Sigma|^2/2$ by \eqref{eq:block-curvature-lower-bound}.
\end{proof}

{
\renewcommand{\bibliofont}{\scriptsize}
\bibliographystyle{plain}
\bibliography{AllenCahn3D_combined}
}
\end{document}